\documentclass[11pt]{article}
\usepackage{amsfonts}
\usepackage{amsmath,amsfonts,amsthm,amscd,amssymb,mathrsfs,setspace, textcomp}
\usepackage{latexsym,epsf,epsfig}
\usepackage{cite}
\usepackage{color}
\usepackage{hyperref}
\usepackage{graphicx}
\usepackage{enumerate}
\usepackage{upgreek}
\usepackage{indentfirst}
\usepackage{empheq}
\usepackage{todo}
\usepackage{lscape}
\usepackage{tikz}
\usetikzlibrary{cd}
\usetikzlibrary{patterns,matrix,arrows,decorations,calc,decorations.pathreplacing,backgrounds,shapes}
\usetikzlibrary{decorations.pathmorphing,decorations.text}
\usetikzlibrary{scopes,decorations.markings,positioning,shapes.arrows}
\usepackage{graphicx}  
\usepackage{setspace}

\usepackage{xcolor}
\usepackage{mathtools}
\mathtoolsset{showonlyrefs}
\usepackage{xcolor}
\def\neweq#1{\begin{equation}\label{#1}}
\def\endeq{\end{equation}}
\def\eq#1{(\ref{#1})}

\newcommand\xit{(\xi,t)}

\newcommand\into{\int_\Omega}
\newcommand\R{\mathbb R}

\newcommand{\ds}{\displaystyle}

\newcommand{\cE}{{\mathcal{E}}}
\newcommand{\bA}{\mathbf A}
\newcommand{\cF}{\mathcal{F}}
\newcommand{\Ez}{E_{z}}

\theoremstyle{plain}
\newtheorem{theorem}{Theorem}[section]
\newtheorem{lemma}[theorem]{Lemma}
\newtheorem{proposition}[theorem]{Proposition}
\newtheorem{definition}[theorem]{Definition}
\newtheorem{corollary}[theorem]{Corollary}
\newtheorem{condition}{Condition}
\theoremstyle{remark}
\newtheorem{remark}[theorem]{Remark}
\numberwithin{equation}{section}
\numberwithin{theorem}{section}
\usepackage{titlesec}
\titlespacing\section{0pt}{4pt plus 3pt minus 3pt}{4pt plus 3pt minus 3pt}
\titlespacing\subsection{0pt}{4pt plus 3pt minus 3pt}{4pt plus 3pt minus 3pt}
\titlespacing\subsubsection{0pt}{4pt plus 3pt minus 3pt}{4pt plus 3pt minus 3pt}

\title{Long-time dynamics of a
hinged-free plate\\
driven by a non-conservative force}

 \author{\normalsize \begin{tabular}[t]{c@{\extracolsep{.8em}}c@{\extracolsep{.8em}}c@{\extracolsep{.8em}}c}
       Denis Bonheure \footnote{Universit\'e Libre de Bruxelles, Belgium}&  Filippo Gazzola \footnote{Politecnico di Milano, Italy}&  Irena Lasiecka \footnote{University of Memphis, Tennessee} &
       Justin Webster \footnote{University of Maryland, Baltimore County, Maryland}\\
\it denis.bonheure@ulb.ac.be & \it filippo.gazzola@polimi.it   &  \it lasiecka@memphis.edu & \it websterj@umbc.edu
\end{tabular}}

\begin{document}
\maketitle

\begin{abstract} {\noindent A partially hinged, partially free rectangular plate is considered, with the aim to address the possible unstable end behaviors of a suspension bridge subject to wind.
This leads to a nonlinear plate evolution equation with a nonlocal stretching active in the span-wise direction. The wind-flow in the chord-wise direction is modeled through a piston-theoretic approximation, which provides both weak (frictional) dissipation and non-conservative forces. The long-time behavior of solutions is analyzed from various points of view. Compact global attractors, as well as fractal exponential attractors, are constructed using the recent quasi-stability theory. The non-conservative nature of the dynamics requires the direct construction of a uniformly absorbing ball, and this relies on the superlinearity of the stretching. For some parameter ranges, the non-triviality of the attractor is shown through the spectral analysis of the stationary linearized (non self-adjoint) equation and the existence of multiple unimodal solutions is shown. Several stability results, obtained through energy estimates under various smallness conditions and/or assumptions on the equilibrium set, are also provided. Finally, the existence of a finite set of determining modes for the dynamics is demonstrated, justifying the usual modal truncation in engineering for the study of the qualitative behavior of suspension bridge dynamics.
\vskip.2cm

\noindent {\bf Key terms}:  nonlinear plate equation, global attractor,  stability, determining modes, non-conservative term, non self-adjoint operator
\vskip.2cm
\noindent {\bf MSC 2010}: 35B41, 35G31, 35Q74, 74K20, 74H40, 70J10}
\end{abstract}
\vfill\eject

{\tableofcontents}

\newpage

\section{Introduction}

We consider the 2-dimensional plate
$\Omega=(0,\pi)\times(-\ell,\ell)$, see the figure below,
\begin{center}
\includegraphics[width=0.5\textwidth]{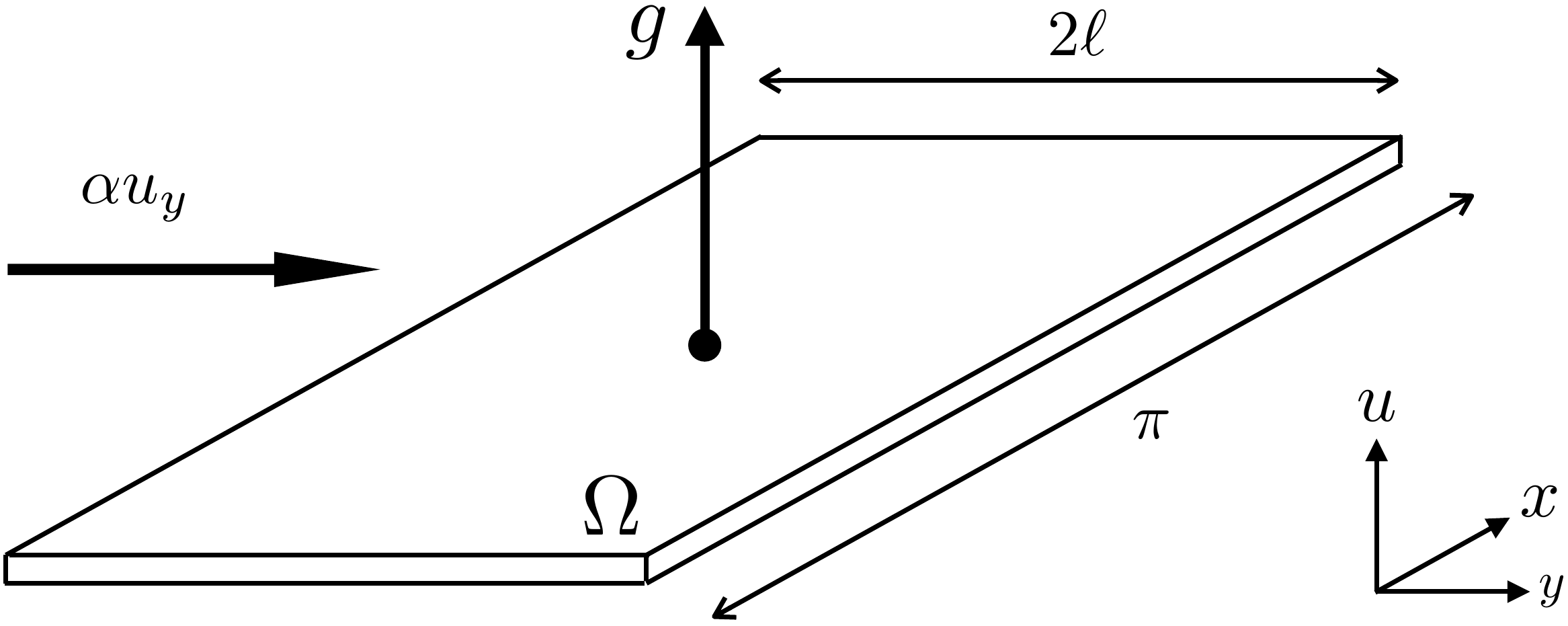}
\end{center}
having two opposed hinged edges ($\{0\}\times[-\ell,\ell]$ and $\{\pi\}\times[-\ell,\ell]\}$), with the remaining two free edges ($[0,\pi]\times \{-\ell\}$ and $[0,\pi]\times \{\ell\}$), governed by
the nonlinear and nonlocal evolution equation
\begin{equation}\label{old}
u_{tt}+\delta u_t + \Delta^2 u + \left[P - S \int_\Omega u_x^2\right] u_{xx}=f\, .
\end{equation}
The constant $\delta\ge0$ measures the (weak) frictional damping, $P\ge0$ represents a longitudinal prestressing while the function $f$
is an external force. A cubic-type nonlinearity naturally arises when large deflections of a beam or a plate are considered while stretching effects suggest the use of variants of the classical Euler-Bernouilli
theory \cite{karman,Knightly}. This explains the nonlocal term in \eqref{old}. 
The equation \eqref{old} was introduced in \cite{bongazmor,FerGazMor} for the analysis, from several points of view, of the stability
of suspension bridges. In this case, $f$ represents the action of a cross-wind and,
the prototype forcing $f$ considered in \cite{bongazmor} was periodic in time, aiming to describe the (periodic) action of the
vortex shedding on the deck of a bridge. Although {direct} aerodynamics effects are neglected in \eqref{old}, the results obtained in
these papers showed a good agreement with the behavior of real bridges: qualitative matching between thresholds
of stability found in \cite{bfg1,bongazmor,FerGazMor} and the one observed in the Tacoma Narrows Bridge disaster \cite{ammvkwoo}.\par
In the present work we take a step towards a force $f$ in \eqref{old} that accounts for  both
aerodynamic forces and damping, such that the resulting equation reads
\begin{equation}\label{veryfirst}
u_{tt}+\delta u_t + \Delta^2 u + \big[P - S \int_\Omega u_x^2\big] u_{xx}=g-\beta[u_t+\mathcal Wu_y]\, .
\end{equation}
The distributed force $f$ now comprises a {time-independent} transverse loading $g$ and  an aerodynamic load modeled by
the so-called {\em piston-theoretic} approximation of an inviscid potential flow \cite{dowell,pist2}. This simple fluid mechanical model is  popular in structural engineering and aeroelasticity because the fluid pressure is incorporated into the structural dynamics with minimal added complexity \cite{pist1}. This aerodynamic approximation is inherently  quasi-steady, as the history of the motion
is neglected in the forcing. Specifically, we model
the flow of the unperturbed wind velocity field $\mathcal W\mathbf e_2$ (the $y$-direction) hitting the plate via the {\em downwash of the flow}, given by $-\beta[u_t+\mathcal Wu_y]$ with $\beta\ge 0$ being a density parameter \cite{pist1}. This is an admittedly crude aeroelastic approximation,
but it is a strikingly simple way to capture both damping and non-conservative flow effects, thereby permitting a study of the dynamic
aeroelastic response of the plate (bridge deck).

Below, Theorem \ref{th:main1} guarantees the existence of a smooth, finite-dimensional global attractor containing the essential asymptotic
behaviors of the dynamics of \eqref{veryfirst}. But this plate equation contains a non-conservative, lower order term that may cause structural self-excitations  \cite{HLW,bolotin,dowell,survey2}. Since $-\beta \mathcal W u_y$ destroys the dynamics' gradient structure, the attractor is not simply characterized as the unstable manifold of the equilibria set.
From the point of view of the non-self-adjoint stationary problem (\eqref{stationaryplate} below), the function $g$ and the parameter $\alpha:=-\beta \mathcal W$ are the key players. Assuming that both are small guarantees the uniqueness of a stationary solution (Theorem \ref{th:main2}). Discarding the smallness assumption,
Theorem \ref{modalsol} asserts the existence of multiple unimodal stationary solutions, whose number grows with the parameter $|\alpha|$, which are, furthermore,  the building blocks for the construction of explicit time dependent unimodal solutions of \eqref{plate}.
These results highlight the possible complexity of the global attractor, providing different behaviors for long-time dynamics. Precisely, the multiplicity of stationary solutions underlies the subtlety and difficulty
of all the results presented here. Theorems \ref{th:main2} and \ref{th:main2a} concern convergence to equilibria from two distinct points of view: the former translates smallness conditions on $\alpha$ and $g$ into stability, while the latter puts hypotheses on the structure of the stationary set, then yielding exponential decay to equilibria.
A further novel aspect of our analysis is given in Theorem \ref{defectcorollary} where we show that a finite number of determining modes for the dynamics associated to \eqref{veryfirst}, allows approximation of the attractor by finitely many ``degrees of freedom". This justifies classical engineering analysis \cite{bleich,smith}. 
Overall, {we establish a rigorous foundation for the end-behaviors of the aeroelastic model \eqref{veryfirst}} {utilizing a variety of techniques ranging from Lyapunov methods, eigenfunction expansions, ODE theory, and the recent quasi-stability theory}.

There is a vast literature studying the aerodynamic responses of a bridge deck under the influence of the wind;
see \cite{flutter1,flutter2,Ro39,bookgaz,larsen98,scanlan,flutter3,smith} and references therein. Most of relevant studies deal with canonical boundary conditions, but the hinged-free conditions we impose here, first suggested in \cite{bfg1,2015ferreroDCDSA}, appear most realistic for modeling bridges. This partially hinged configuration helps, yielding the expected regularity for associated elliptic solvers \cite{bfg1,bongazmor,FerGazMor,2015ferreroDCDSA}. 
\par

While the long-time behavior of nonlinear elastic structures forced by external/internal inputs has been under investigation for many years, the model \eqref{veryfirst} displays a number of features that result from terms which are indispensable
for accurate wind-bridge interaction modeling. Navigating the delicate balance between aerodynamic damping and destabilizing non-conser\-vative terms is central in our analysis, and distinct from most literature addressing gradient dynamics, e.g., \cite{kalzel,fgmz,pata1}, save for \cite{springer,HLW} but traditional plate boundary conditions are imposed therein. 
Related dynamical systems analyses are largely based on the property of {dissipativity}, in
the sense that the system energy is non-increasing, which is precisely {not the case for \eqref{veryfirst}}
since the force depends on the solution and yields an energy-building contribution. 
This precludes the use of shelf-ready tools in dynamical system where, for instance, existence of a global attractor is reduced to demonstrating one asymptotic compactness property. Here, a string of estimates exploiting the geometry of $\Omega$, the boundary conditions, and the structure of the nonlinearity are utilized to  establish the existence of an absorbing ball (Proposition \ref{absorbingball}), despite the presence of non-conservative terms.  The difficulty in constructing the attractor is also depicted by the surprising multiplicity of stationary solutions.

\section{Functional setting and well-posedness}\label{energiessec}
In 1950, Woinowsky-Krieger \cite{woinowsky} modified the classical beam models of Bernoulli and Euler by assuming a nonlinear dependence of the axial strain on the deformation gradient that accounts for
 stretching in the beam due to elongation (extensibility). This leads to the elastic energy
  $$
\frac{1}{2}\int_Iu_{xx}^2+\frac{S}4\left(\int_Iu_x^2\right)^2\!\!,
$$
where $I$ is the interval representing the beam at equilibrium and $S \ge 0$ indicates the strength of the restoring force resulting from
stretching in $x$. 
Thus the aforementioned nonlocal stretching effect is only noted in the $x$-direction, which gives rise to the superquadratic energy 
$\frac{S}4\left(\int_\Omega u_x^2\right)^2$.

To simplify notation, we consider an overall damping $k=\delta+\beta$ (accounting for imposed and aerodynamic damping), and a generalized flow parameter
$-\beta \mathcal W:=\alpha \in \mathbb R$. We take {longitudinally hinged, laterally free} boundary conditions with
Poisson ratio $\sigma \in (0,1)$. The system, in strong form,
\begin{align}\label{plate}
\left\{
\begin{array}{rl}
u_{tt}+k u_t + \Delta^2 u + \left[P - S \int_\Omega u_x^2\right] u_{xx}=g+\alpha u_y &\textrm{in }\Omega\times(0,T)\\
u = u_{xx}= 0 &\textrm{on }\{0,\pi\}\times[-\ell,\ell]\\
u_{yy}+\sigma u_{xx} = u_{yyy}+(2-\sigma)u_{xxy}= 0 &\textrm{on }[0,\pi]\times\{-\ell,\ell\}\\
u(x,y, 0) = u_0(x,y), \quad \quad u_t(x,y, 0) = v_0(x,y) &\textrm{in }\Omega\,,
\end{array}
\right.
\end{align}
is the one on which we will operate, and our main results in Section \ref{maintheorems} will be phrased in terms
of the constants in \eqref{plate}. For suspension bridges, the prestressing parameter $P$ is typically taken in the range $0 < P < \lambda_2$, namely below the second eigenvalue of the principal structural operator defined below in \eqref{Adef}: the range $0\le P <\lambda_1$ (the first eigenvalue) is usually called {weakly prestressed} whereas the range $\lambda_1\le P<\lambda_2$ is called {strongly prestressed} for plate
equations with these boundary conditions \cite{chugg,bbfg}. We deal mostly with a weakly prestressed plate when considering stability
(see Theorems \ref{th:main2} and \ref{th:main2a}), though some of our main results allow $P \in \mathbb R$ (Theorems \ref{th:main1}, \ref{modalsol}, and \ref{defectcorollary}).\par
We denote by $H^s(\Omega)$
the Hilbert Sobolev space of order $s \in \mathbb R$ with norm $||\cdot||_s$. We write the inner product in $L^2(\Omega)$ as $(\cdot,\cdot)$.
The notation $B_R(X)$ will be used for the open ball in $X$ of radius $R$ centred at $0$.
The phase space of admissible displacements for the hinged-free plate \eqref{plate} is
\begin{align}
H^{2}_*=\{u \in H^2(\Omega); \ u=0 \ \textrm{on }\{0,\pi\}\times[-\ell,\ell]\}\,,
\end{align}
and its dual is denoted by $(H^{2}_*)'$.
The  brackets $\langle{\cdot, \cdot}\rangle$  denote the duality pairing between $(H^{2}_*)'$ and  $H^{2}_*$.
Following \cite[Lemma 4.1]{2015ferreroDCDSA}, for $\sigma\in(0,1)$, we equip $H^2_*$ with the scalar product
\begin{align}\label{biform} a(u,v) := \!\int_{\Omega}\left( \Delta u \Delta v\!-\!(1\!-\!\sigma)\big[u_{xx}v_{yy}\!+\!u_{yy}v_{xx}\!-\!2u_{xy}v_{xy}
\big] \right), \quad u,v\in H^{2}_*(\Omega)\,,
\end{align}
which induces a norm ~$||u||_{H^2_*} = \sqrt{a(u,u)}$ equivalent to the usual Sobolev norm $||\cdot||_2$.

The phase space for the dynamics will be denoted by ~$Y:= H_*^2  \times L^2(\Omega),$~
with inner product and norm given respectively by
$$(y_1,y_2)_Y = \big((u_1,v_1),(u_2,v_2)\big)_Y = a(u_1,u_2)+(v_1,v_2) ~\text{ and } ~ ||y_1||_Y^2 = ||u_1||^2_{H^2_*}+||v_1||^2_0.$$
We then define the positive, self-adjoint biharmonic operator corresponding to ~$a(\cdot,\cdot)$, taken with the boundary
conditions in \eqref{plate}: $A: L^2(\Omega) \to L^2(\Omega)$ is given by $Au=\Delta^2u$ with
\begin{align}\label{Adef}\mathcal D(A) =\big\{ u \in H^4(\Omega) \cap H^2_*~:&~u_{xx}=0~~\textrm{on}~~\{0,\pi\}\times [-\ell,\ell],\\ &~u_{yy}+\sigma u_{xx} = u_{yyy}+(2-\sigma)u_{xxy}= 0 ~~\textrm{on } ~[0,\pi]\times\{-\ell,\ell\}
  \big\}.\end{align}
Observe that $u_{xx}= 0$ on $\{0,\pi\}\times[-\ell,\ell]$  and $u_{yy}+\sigma u_{xx} = u_{yyy}+(2-\sigma)u_{xxy}= 0$ on $[0,\pi]\times\{-\ell,\ell\}$
are  the {natural boundary conditions} associated with $a(\cdot,\cdot)$ in its strong form. The spectral theorem provides the eigenvalues of $A$ on $L^2$; these are discussed at length in Appendix \ref{spectral} and denoted by $\{\lambda_j\}_{j=1}^{\infty},$ noting that $\lambda_1$ will be used frequently in the discussions below.

{Strong solutions} satisfy the PDE in \eqref{plate} the pointwise sense and correspond to initial data $(u_0,v_0) \in \mathcal D(A) \times H^2_*$, i.e., in the domain of the generator for the linear plate equation.
{Generalized solutions} are  $C^0([0,T]; Y)$ limits of strong solutions and such solutions correspond to initial data taken in the state space $(u_0,v_0) \in Y$.
Finally, weak solutions satisfy a variational formulation of \eqref{plate} almost everywhere in $t$; we provide the precise definition thereof for the sake of clarity.
\begin{definition}[Weak solution]\label{df:weaksolution}
Let {$g\in L^2(\Omega)$}. A weak solution of \eqref{plate} is a function
\begin{empheq}{align}
u\in C^ 0(\R_+,H^{2}_*(\Omega))\cap C^1(\R_+,L^2(\Omega))\cap C^2(\R_+,(H^{2}_*(\Omega))')
\end{empheq}
such that for all $t>0$ and all $v\in H^{2}_*(\Omega)$, one has
\begin{empheq}{align}\label{weakform}
\langle u_{tt},v \rangle+ k (u_t,v) + a(u,v) +\big[S\|u_x\|_{0}^2-P\big](u_x,v_x)= (g,v)+\alpha(u_y,v).
\end{empheq}
\end{definition}
\noindent Strong solutions are generalized, and generalized solutions are weak \cite{springer}.

Following \cite{bongazmor,springer,FerGazMor},  we introduce plate energies for mixed-type boundary conditions
\begin{empheq}{align}\label{energies}
E(t) =\dfrac{1}{2}\Big[a(u(t),u(t)) +\|u_t(t)\|_0^2\Big]\ \text{ and }\ \mathcal E(t) = E(t)+\dfrac{S}{4} \|u_x(t) \|_0^4-\dfrac{P}{2}\|u_x(t)\|_0^2-{(g,u(t))}.
\end{empheq}
It is also useful to emphasize the nonnegative part of the energy
\begin{equation}\label{E+}
E_+(t) = E(t)+ \dfrac{S}{4} \|u_x(t)\|_0^4.
\end{equation}
When the context is clear we will use the notations ~$E,~ E_+,~\mathcal E$, suppressing the time dependence.

The following well-posedness result follows from combining \cite{bongazmor,FerGazMor} and \cite[Section 4.1.1, p.197]{springer}.

\begin{proposition}\label{p:well} Assume that $k \ge0$, $P \in \mathbb R$, $S>0$, $\alpha \in \mathbb R$ and $g\in L^2(\Omega)$.
For any initial data
$y_0 = (u_0,v_0) \in \mathcal D(A) \times H_*^2$, the problem \eqref{plate} has a unique strong solution.

For any initial data
$
y_0=(u_0,v_0) \in Y,
$
the problem \eqref{plate} has a unique generalized (and hence, weak)
solution $u(t)$. We denote the associated $C_0$ semigroup by $(S_t,Y)$, where
$S_ty_0 = (u(t),u_t(t))$ is the strong solution to \eqref{plate} for $y_0 \in \mathcal D(A) \times H^2_*$ and the generalized solution when $y_0 \in Y$.

Any weak solution satisfies, for $0\le s<t$, the energy identity
\begin{equation}\label{energyrelation}
{\mathcal E(t)+k\int_s^t \|u_t(\tau)\|_{0}^2d\tau=\mathcal E(s)+\alpha \int_s^t \big(u_y(\tau),u_t(\tau)\big) d\tau}.
\end{equation}
If $B_R(Y)$ is an invariant set under $S_t$, then there exists $a_0(R),\omega_0(R)>0$ such that:
\begin{equation}\label{dynsys}
\|S_t(y_1)-S_t(y_2)\|_Y^2\le a_0 e^{\omega_0 t} \|y_1-y_2\|_Y^2,~~\forall y_1,y_2 \in B_R(Y).
\end{equation}
\end{proposition}

\begin{remark}
In Section \ref{ball} we show that the dynamical system admits an absorbing ball, which shows that invariant sets $B_R(Y)$ exist, thereby
obtaining the Lipschitz property stated in \eqref{dynsys}. \end{remark}

Thanks to the spectral decomposition (see Appendix \ref{spectral}), we can
write any solution $u$ of \eqref{plate}
\begin{equation}\label{Fourier}
u(\xi,t)=\sum_{i=1}^\infty h_i(t)w_i(\xi)\, ,
\end{equation}
so that $u$ is identified by its Fourier coefficients ($g \equiv 0$) which satisfy the infinite-dimensional system
\begin{equation}\label{infsystem}
\ddot{h}_j(t)+k\dot{h}_j(t)+\lambda_jh_j(t)+m_j^2\left[S\sum_{i=1}^\infty m_i^2h_i(t)^2-P\right]h_j(t)=\alpha \sum_{i=1}^\infty  \Upsilon_{ij} h_i(t),~j = 1,2,3,...
\end{equation}
with $m_j$ the frequency in the $x$-direction and $\displaystyle \Upsilon_{ij}= ( \partial_y w_i, w_j)$. 
From \eqref{infsystem}, we see modal coupling via the nonlocal and non-conservative terms. Since the modes are either even or odd in the direction $y$, see Appendix \ref{spectral}, the non-conservative force induces a coupling between modes of opposed parity. In particular, contrary to \cite{bongazmor}, it induces a coupling between vertical and torsional modes.

\section{Main results and discussion}\label{maintheorems}

\subsection{Attractors and stability}

Our results make use of dynamical systems terminology (details are found in Appendix \ref{appendix3}).
The {fractal} {dimension} of a set $A \subset Y$ is defined by
\[
\text{dim}_f A=\limsup_{\epsilon \to 0}\frac{\ln n(A,\epsilon)}{\ln (1/\epsilon)}\;,
\]
where $n(M,\epsilon)$ is the minimal number of balls ~$y_i+B_{\epsilon}(Y)$ whose closures cover the set $M$.\par
We recall that for the dynamics $(S_t,Y)$,
a compact {global attractor} $\mathbf{A} \subset \subset Y$ is an invariant set ($S_t(\mathbf{A})=\mathbf{A}$ for all
$t \in \mathbb R$) that uniformly attracts bounded  $B \subset Y$:
\begin{equation}  \label{dist-u}
\lim_{t\to+\infty}d_{Y}\{S_t (B)~|~\bA\}=0,~~~\mbox{where}~~
d_{Y}\{S_t (B)~|~\bA\}\equiv\sup_{y\in S_t B}{\rm dist}_{Y}(y,\bA).
\end{equation}
	
A {generalized\footnote{The word ``generalized'' is included to indicate that the finite-dimensionality requirement is allowed in a  topology weaker than $Y$. See \cite{springer,quasi}.} fractal exponential attractor} for  $(S_t,Y)$ is a forward invariant, compact set, $A_{\text{exp}} \subset \subset Y$, having finite fractal dimension, that attracts bounded sets (as above) {with uniform exponential rate} in $Y$.
Our first theorem concerns attractors for $(S_t,Y)$ associated to \eqref{plate}.

\begin{theorem}[Attractor]\label{th:main1}
Assume that $k > 0$, $P \in \mathbb R$, $S>0$, $\alpha \in \mathbb R$ and $g\in L^2(\Omega)$.
 There exists a compact global attractor $\mathcal A$ for the dynamical system $(S_t,Y)$ corresponding to generalized solutions to \eqref{plate} as in Proposition \ref{p:well}. Moreover,

\noindent$\bullet$ it is smooth in the sense that $\mathcal A \subset (H^4\cap H^2_*) \times H_*^2$ and is a bounded set in that topology;

\noindent$\bullet$ it has finite fractal dimension in the space $Y$;

\noindent$\bullet$ there exists a generalized fractal exponential attractor $\widetilde{\mathcal A}_{\text{exp}}\subset \subset Y$, with finite fractal dimension in $\widetilde Y := L^2(\Omega) \times (H_*^2)'$.
\end{theorem}

Note that Theorem \ref{th:main1} does not require any {imposed} damping in \eqref{plate}, that is, if $\delta=0$ in \eqref{veryfirst}, $\beta>0$ implies $k>0$. This shows that the aerodynamic damping in the model is sufficient to yield an attractor given any flow   $\alpha \in \mathbb R$ and any pre-stressing $P \in \mathbb R$. We remark that the superlinear restoring force is essential for the existence of a uniform absorbing ball in this general setting; see Section \ref{ball}.
In the sequel,
we  focus our attention on the weakly prestressed case, that is, $0\le P<\lambda_1$, where $\lambda_1$ is the least eigenvalue of $A$ in \eqref{Adef}, given explicitly in \eqref{lambda}.

In what follows, the set of stationary solutions of \eqref{plate} plays a major role.  We have the following result, whose proof is given in Section \ref{pointstab}.
\begin{theorem}[Stability I]\label{th:main2}
Let $S>0$, $0\le P<\lambda_1$ be given. For any $k>0$, there exists $C(k)>0$ such that if ~$\|g\|_0+|\alpha|<C(k)$, then
\eqref{plate} admits a unique stationary solution $u_g$. Moreover,\par\noindent
$\bullet$ all solutions $(u(t),u_t(t))$ to \eqref{plate} converge (uniformly) exponentially to $(u_g,0)$ in $Y$ as $t\to\infty$;\par\noindent
$\bullet$ $u_g=0$ is the unique stationary solution provided that
$$g=0~\text{ and }\ |\alpha| < \frac{\lambda_1-P}{\sqrt{\lambda_1}}\sqrt{2(1-\sigma^2)}\, .$$
\end{theorem}

In the next statement, we present a second stabilization result from a  different perspective. This result places the emphasis on hypotheses on the set of stationary solutions, but yields less precise estimates than those supporting Theorem \ref{th:main2}. In particular, the possibility of multiple equilibria is permitted, but the proof of the latter result is  rooted in a control and dynamical systems approach.

Let $W$ denote the set of stationary solutions to \eqref{plate} -- namely, the weak solutions to \eqref{stationaryplate} --
as described in more detail in Section \ref{secstatsols}.
\begin{theorem}[Stability II]\label{th:main2a}
Let $S > 0$, $0 \leq P < \lambda_1$.
\par
\noindent $\bullet$ If $g=0$ and $0$ is the unique element of $W$, there exists $q>0$ so that if  ~$k^{-1}{\alpha^2} \le q$, then all solutions $(u(t),u_t(t))$ to \eqref{plate} converge (uniformly) exponentially
to $(0,0)$ in $Y$ as $t\to\infty$;\par\noindent
$\bullet$ If $W=\{e\}$ and the singleton $e$ is also hyperbolic as an equilibrium, then there exists $q_e>0$ so that if  ~$k^{-1}{\alpha^2} \le q_e$,
all corresponding solutions $(u(t),u_t(t))$ to \eqref{plate} converge (uniformly) exponentially to $(e,0)$ in $Y$ as $t\to\infty$;\par\noindent
$\bullet$ If $W$ consists only of isolated, hyperbolic equilibria, then for any solution $(u(t),u_t(t))$ to \eqref{plate} that converges
to an equilibrium $(e,0)$ in $Y$ as $t\to\infty$, there exists $q_e>0$ so that if  ~$k^{-1}{\alpha^2} \le q_e$,  then the convergence is exponential in $Y$, with a rate that depends on: $e$, $q_e$, and the trajectory itself.
\end{theorem}
\begin{remark}\label{stabilization}
The hyperbolicity assumption is guaranteed if we impose smallness of $\alpha$. Indeed
taking the inner product with $v$ above yields
\begin{equation}\label{v3}
a(v,v)- P\|v_x\|_0^2 + S\|e_x\|_0^2\|v_x\|_0^2 + 2 (v_x, e_{x} )^2= \alpha (v_y,v)
\end{equation}
Invoking coercivity w.r.t. $P$, the above equation has zero solution provided $\alpha $ is small enough.

A corollary to the proof of Theorem \ref{th:main2a} is that if we have a trajectory in hand $(u,u_t)$ that converges strongly to a known (isolated, hyperbolic) equilibrium point in $Y$, then the convergence rate is exponential. When $g \neq 0$ and is potentially large, and/or when $\alpha$ is large, Theorem \ref{th:main2} can still provide an exponential rate of convergence, if, a priori, it is know that a trajectory is converging to equilibria. Compare with Theorem \ref{modalsol} and see Remark \ref{connecting}.
For these stabilization results, the essential ingredients are smallness of ${\alpha^2}{k}^{-1}$,~ $g$, and $\alpha$. The rates of convergence are exponential regardless of the damping value $k >0$, although if one wishes to control the rate of exponential convergence,  decreasing $\alpha$ and $g$ or the addition of {static damping} would be required.
If $u_g$ is the unique equilibrium point that happens to be hyperbolic, then Theorem \ref{th:main2a} recovers the result from
Theorem \ref{th:main2}. The first part of Theorem \ref{th:main2a} mirrors the second part of Theorem \ref{th:main2}, but the hypothesis on
the smallness of $g$ (that yields a unique $u_g$) is replaced by the assumptions of uniqueness and hyperbolicity of the equilibrium point $e$.
\end{remark}

\subsection{Non-triviality of the attractor}\label{secstatsols}

As shown by Theorem \ref{th:main2} and Remark \ref{stabilization}, the attractor may, in some cases, reduce to the unique stationary point, in which case it can be considered
``trivial''. Any further characterization of the attractor obtained in Theorem \ref{th:main1} requires knowledge of the number of stationary solutions of \eqref{plate}, namely solutions to the problem:
\begin{align}\label{stationaryplate}
\left\{
\begin{array}{rl}
 \Delta^2 u  - S ||u_x||_0^2u_{xx}+Pu_{xx}-\alpha u_y=g&\textrm{in }\Omega\times(0,T)\\
u = u_{xx}= 0 &\textrm{on }\{0,\pi\}\times[-\ell,\ell]\\
u_{yy}+\sigma u_{xx} = u_{yyy}+(2-\sigma)u_{xxy}= 0 &\textrm{on }[0,\pi]\times\{-\ell,\ell\}.
\end{array}
\right.
\end{align}
In general, one should expect multiple solutions to \eqref{stationaryplate}, see \cite{ciarlet}. The first statement presented here shows that finite multiplicity of solutions is a rather generic property.

\begin{theorem}\label{generic}
Let $g\in L^2(\Omega)$ with $S >0$ and $P,\alpha \in \mathbb R$. Then:\par
\noindent $\bullet$ Problem \eqref{stationaryplate} has a weak solution. Any solution is a strong limit of a Galerkin approximation.\par
\noindent$\bullet$ The set $W$ of all weak solutions of \eqref{stationaryplate} is contained in $\mathcal D(A)$ (as in \eqref{Adef}).\par
\noindent $\bullet$ There exists an open dense set $\mathscr R \subset L^2(\Omega)$ such that if $g \in \mathscr R$ then $W$ is a finite set.
\end{theorem}

We omit the proof of Theorem \ref{generic} since it can be obtained as in \cite[Theorem 1.5.9]{springer}, adapted to the configuration
of our boundary conditions on $\partial\Omega$. The argument utilizes pseudomonotone operator theory and rests on the infinite-dimensional Sard-Smale Theorem. The proof of Theorem \ref{generic} {critically requires} Lemma \ref{potentiallowerbound} below which holds for the configuration at hand.\par

Above, Theorem \ref{generic}, states that the set of solutions to \eqref{stationaryplate} is ``well-behaved'' but it does not directly assert the existence of multiple solutions. In order to {prove} that the attractor can be ``complex'' (in particular, not reduced to a single stationary
solution), we take $g\equiv 0$ (for simplicity) and seek solutions of \eqref{plate} (resp.\ of \eqref{stationaryplate}) of the form
\begin{equation}\label{unimodsol}
V_{m,\alpha}(x,y,t)=\phi(t)\psi(y)\sin(mx)\qquad\Big(\mbox{resp. }U_{m,\alpha}(x,y)=\psi(y)\sin(mx)\Big)
\end{equation}
for some integer $m$. When $\phi,\psi\neq0$ these solutions will be referred to as {\bf unimodal solutions},
by analogy with \eqref{Fourier}. Such solutions count the number of zeros ($m-1$) in the $x$-direction and, obviously, depend on $\alpha$. The following result is proved in Section \ref{modal}.

\begin{theorem}\label{modalsol} Assume that $P\in \R$ and that $g=0$.
For any integer $m$ there exists $\overline{\alpha}_m<0$ such that for all $\alpha<\overline{\alpha}_m$, the following assertions hold:\par
\noindent$\bullet$ There exists a unimodal solution $U_{m,\alpha}$ of \eqref{stationaryplate}, see \eqref{unimodsol}, having $m-1$ zeros in
the $x$-direction.\par
\noindent$\bullet$ There exists at least $m$ unimodal solutions $U_{1,\alpha},...,U_{m,\alpha}$ of \eqref{stationaryplate}; these solutions have from $0$
up to $m-1$ zeros in the $x$-direction.\par
\noindent$\bullet$ There exist infinitely many unimodal solutions $V_{m,\alpha}$ of \eqref{plate}, of the kind~$\phi(t)U_{m,\alpha}$ as in \eqref{unimodsol},
and, as $t\to\infty$,
$$
\text{either}\quad V_{m,\alpha}(t)\to0\quad \text{ or } \quad V_{m,\alpha}(t)\to U_{m,\alpha}\quad \text{ or } \quad V_{m,\alpha}(t)\to -U_{m,\alpha}.
$$
The initial data in \eqref{plate} may be chosen in such a way so that $V_{m,\alpha}(t)$ attains any of these limits.
\end{theorem}

{\begin{remark}\label{connecting} Compare the above result to the third bullet point of Theorem \ref{th:main2a}; if further characterization of the equilibria set $W$ is available (namely, if we know that the $U_{m,\alpha}$ are isolated as elements of $W$), we would conclude that the rate of decay in Theorem \ref{modalsol} is exponential.
\end{remark}}

As the generalized flow parameter $\alpha$ decreases towards $-\infty$, one has that $\alpha<\overline{\alpha}_m$ for an increasing number of integers $m$; hence we have the following consequence of Theorem \ref{modalsol}.

\begin{corollary}
If $\alpha\to-\infty$, the number of solutions $U_{m,\alpha}$ tends to infinity.
\end{corollary}

We also point out that  the same results hold when $\alpha>0$, as only the size of $|\alpha|$ matters, see Remark \ref{yto-y}. Both
$\pm \alpha$ are physically relevant, corresponding to Northward versus Southward flow (so to speak). Here, we only dealt with
$\alpha<0$ in order to discuss winds of given direction: then $\alpha=-\beta \mathcal W$ where $\beta>0$ is the aerodynamic density coefficient
and $\mathcal W>0$ is the flow velocity.

\begin{remark} For the proof of Theorem \ref{modalsol} we need to study the {extended-type}
eigenvalue problem:
$$
\Delta^2 u -\mu u_{xx}= \alpha u_y\qquad\textrm{in }\Omega,\quad u\in H^2_*,
$$
see \eqref{eigenvalue} below. The existence of real eigenvalues $\alpha$ is not obvious at all. To see this, consider the following
comparable problem posed in $H^2_0(\Omega)$:
\begin{equation}\label{eigenDirichlet}
\Delta^2 u-2\sqrt[3]{\alpha}(\Delta u)_y -\mu u_{xx}= \alpha u_y\qquad\textrm{in }\Omega\, .
\end{equation}
Notice that
\eqref{eigenDirichlet} holds if and only if~
$
{\rm div}\Big[\nabla(e^{\sqrt[3]{\alpha}\, y}\Delta u)-\sqrt[3]{\alpha^2}e^{\sqrt[3]{\alpha}\, y}\nabla u\Big]-\mu e^{\sqrt[3]{\alpha}\, y}\, u_{xx}=0\, .
$
Multiplying this identity by $u$ and integrating several times by parts over $\Omega$, we get
$$
\int_\Omega e^{\sqrt[3]{\alpha}\, y}\Big[(\Delta u)^2+\sqrt[3]{\alpha^2}|\nabla u|^2+\mu u_x^2\Big]=0
$$
which shows that $u=0$ for any $\mu\ge0$ and any $\alpha\in\R$. 
The same example works  under Navier boundary conditions $u=\Delta u=0$ on $\partial\Omega$.
\end{remark}

\subsection{Existence of determining modes}

A common procedure in classical engineering literature is to restrict the attention to a finite number of modes ({modal truncation}). Bleich-McCullough-Rosecrans-Vincent \cite[p.23]{bleich} write that: \begin{quote} out of the infinite number of possible modes of motion in which
a suspension bridge might vibrate, we are interested only in a few, to wit: the ones
having the smaller numbers of loops or half waves.\end{quote} The justification of this approach has physical roots: Smith-Vincent \cite[p.11]{smith}
note that {the higher modes with their shorter waves involve sharper curvature in the truss and, therefore, greater bending moment at
a given amplitude and accordingly reflect the influence of the truss stiffness to a greater degree than do the lower modes}.
Whence, we are interested in analyzing a finite number of modes, provided that these suitably describe the entire dynamics of
\eqref{plate}. From a mathematical point of view, this finite-dimensional approximation is the heart of the classical Galerkin procedure.

One can go one step further mathematically by showing that a finite number of modes $\{e_j\}_{j=1}^N$ (eigenfunctions of $A$ associated to the eigenvalues $\{\lambda_j\}$ as in Proposition \ref{spectrum})  is sufficient to asymptotically describe the dynamics of \eqref{plate}.
Then, for the set of {modal integration} functionals
\begin{empheq}{equation}\label{thesemodes}\mathscr L=\{ l_j ~:~l_j(w)=(w,e_j),~~j=1,...,N\},\end{empheq}
 the Fourier approximation $R_{\mathscr L}:H_*^2\to H_*^2$ by
$$R_{\mathscr L}(w)=\sum_{j=1}^N l_j(w)e_j$$ asymptotically determines the solution.

\begin{theorem}\label{defectcorollary}
Assume that $k > 0$, $P \in \mathbb R$, $S>0$, $\alpha \in \mathbb R$ and $g\in L^2(\Omega)$. Let $\{e_j\}_{j=1}^{\infty}$ be the eigenfunctions of $A$ on $H_*^2$. There exists $N>0$ such that if $y^i=(u^i,v^i)\in Y$, $i=1,2$ solve \eqref{plate} and satisfy
$$\lim_{t\to \infty}\big((u^1-u^2)(t),e_j\big) = 0,\quad \text{for }j=1,...,N,$$
then ~$\ds\lim_{t\to \infty}\|y^1(t)-y^2(t)\|_Y = 0$.
	\end{theorem}
\noindent Theorem \ref{defectcorollary} follows from a more general statement about {determining functionals}  in Section \ref{construct}.	
	
\section{Preliminary results}

\subsection{Energy bounds}

In this section we first introduce and analyze the family of parametrized Lyapunov functions:
\begin{equation}\label{lyapunov}
V_\nu(t):=\frac12 \|u_t(t)\|_{0}^2 +\frac12 a(u(t),u(t)) -\frac{P}{2}\|u_x(t)\|_{0}^2+\frac{S}{4}\|u_x(t)\|_{0}^4+\nu
\into u\xit u_t\xit \, d\xi,\end{equation}
where $\nu>0$, and we derive bounds for $V_{\nu}$. We note that, via an elementary calculation using \eqref{embedding} and Young's inequality,
we have
\begin{equation}\label{positiveVnu}
\nu\le\sqrt{\lambda_1-P}\ \Longrightarrow\ V_\nu \ge 0.
\end{equation}
We take \eqref{positiveVnu} to be a standing hypothesis for $\nu$ to ensure positivity of the Lyapunov function.
We also notice that we can write \eqref{lyapunov} as
\begin{equation}
V_{\nu}(t)  = \mathcal E(t)+\nu \big(u(t),u_t(t)\big),
\end{equation}
where $\mathcal E$ was defined in \eqref{energies}. For every $v\in H^2_*$, we have
\begin{empheq}{align}
\|v\|_{H^2_*}^2=a(v,v)=\int_{\Omega}\left( v_{xx}^2+v_{yy}^2+2(1-\sigma)v_{xy}^2  +2\sigma v_{xx}v_{yy}\right)d\xi \ge 2(1-\sigma^2) \int_{\Omega} v_{yx}^2\, d\xi,
\end{empheq}
since $\sigma<1$. This shows that the inequality
\begin{equation}\label{embedding-H^2_*-u_y}
 \|v_{y}\|_{0}^2 \le  \|v_{yx}\|_{0}^2\le \frac{1}{2(1-\sigma^2)}\|v\|_{H^2_*}^2\qquad\forall v\in H^2_*
\end{equation}
holds.
Observe also that
\begin{equation}\label{estimate-grad-u_y}
\|v\|_{H^2_*}^2
\ge (1-\sigma)\int_{\Omega}\left( (1+\sigma)v_{yy}^2+2v_{xy}^2\right)d\xi\ge (1-\sigma^2)\|\nabla v_{y}\|_0^2.
\end{equation}

Before starting, let us rigorously justify once at the outset the computations that follow for weak solutions. The regularity of weak solutions {does not allow} one
to take $v=u_t$ in \eqref{weakform}. Therefore, we must justify differentiation of the energies $V_\nu$, a computation used extensively below.
In this respect, let us recall a general result, see \cite[Lemma 4.1]{temam}.
\begin{lemma}\label{justification}
Let $(V,H,V')$ be a Hilbert triple. Let $a(\cdot,\cdot)$ be a coercive bilinear continuous form on $V$, associated with the continuous isomorphism $A$ from $V$ to $V'$ such that
$a(u,v)=\langle Au,v\rangle$ for all $u,v\in V$. If $w$ is such that
$$w\in L^2(0,T;V)\, ,\quad w_t\in L^2(0,T;H)\, ,\quad w_{tt}+Aw\in L^2(0,T;H)\, ,$$
then, after modification on a set of measure zero, $w\in C^0([0,T],V)$, $w_t\in C^0([0,T],H)$ and, in the sense of distributions on $(0,T)$,
$$\langle w_{tt}+Aw,w_t\rangle=\frac12 \frac{d}{dt}\big(\|w_t\|_{0}^2+a(w,w)\big)\, .$$
\end{lemma}
For generalized solutions, the requisite multiplier calculations here and in the remainder of the paper can be done in a proper sense on {strong solutions} with {smooth initial data} in $\mathcal D(A) \times H_*^2$; resulting inequalities for generalized solutions are then obtained in the standard fashion via the extension through density of $\mathcal D(A) \times H_*^2$ in $Y$.

\begin{lemma}\label{energybound-PDE}
Assuming that $0\le P<\lambda_1$, $k>0$ and that $u$ is a solution of \eqref{plate}. If ~$0<\nu\le \frac k2$, ~$0<\delta<\frac{k-\nu}2$,~
$\nu^2 \leq \lambda_1-P $ (recall \eqref{positiveVnu}), and
\begin{equation}\label{boundalpha}
\alpha^2\le  \frac{4\delta(1-\sigma^2)\nu(\lambda_1-P-\nu k +\nu^2)}{\lambda_1},
\end{equation}
then
\begin{equation}\label{Vnulimsup}
V_\nu(t)\le{\rm e}^{-\nu(t-t_0)} V_\nu(t_0) + \frac{1-{\rm e}^{-\nu(t-t_0)}}{2\nu(k-\nu-2\delta)} \|g\|_{0}^2,\quad
V_\nu(\infty):=\limsup_{t\to\infty}V_\nu(t)\le \frac{\|g\|_{0}^2}{2\nu(k-\nu-2\delta)}.
\end{equation}
\end{lemma}
\begin{proof} Take any $\nu\in(0,\frac23  k)$. From the definition of $V_\nu$ and, by using Lemma \ref{justification} and \eqref{plate}, we infer
$$\begin{array}{rl}
\dot V_\nu(t)+ \nu  V_\nu(t) = &\!\!\!\!\! \displaystyle \left(\frac{3\nu}2 - k\right)\|u_t(t)\|_{0}^2 -\frac\nu{2}\|u(t)\|_{H^2_*}^2
+\frac{\nu P}{2}\|u_x(t)\|_{0}^2-\frac{3S\nu}{4}\|u_x(t)\|_{0}^4 \\
\ &\!\!\!\!\! \displaystyle+ \nu(\nu- k)\into u\xit u_t\xit \, d\xi + \into (g\xit+\alpha u_y\xit) \big(u_t\xit +\nu u\xit \big)\, d\xi\, .
\end{array}$$

Hence, by using \eqref{embedding}, \eqref{embedding-H^2_*-u_y} and the Young inequality, we obtain for any $\gamma,\delta>0$,
\begin{align}
\dot V_\nu(t)+ \nu  V_\nu(t) \le & \displaystyle\left(\frac{3\nu}2- k+\gamma+\delta\right)\|u_t(t)\|_{0}^2 -\frac\nu{2\lambda_1}\left(\lambda_1-P-2\nu\gamma-2\nu\delta\right)\|u(t)\|_{H^2_*}^2 \label{dif-ineq-energy} \\
\ & \displaystyle +  \frac{1}{4\gamma}\|g\|_{0}^2 +\frac{\alpha^2}{8\delta(1-\sigma^2)}\|u(t)\|_{H^2_*}^2 +\nu(\nu- k+2\gamma+2\delta)\into u\xit u_t\xit \, d\xi\, .
\end{align}
To get a global estimate, we 
choose $\gamma+\delta=\frac{k-\nu}2$ and $\nu\le k/2$, so that
$$
\dot V_\nu(t)+ \nu  V_\nu(t) \le \frac{\lambda_1\alpha^2-4\delta(1-\sigma^2)\nu(\lambda_1-P-\nu k +\nu^2)}{8\lambda_1\delta(1-\sigma^2)}\|u(t)\|_{H^2_*}^2 + \frac{\|g\|_{0}^2}{2(k-\nu-2\delta)},
$$
where $\delta>0$. Now we see that, if \eqref{boundalpha} holds, then
$$
\dot V_\nu(t)+ \nu  V_\nu(t) \le \frac{\|g\|_{0}^2}{2(k-\nu-2\delta)}.
$$
Finally, for all $t_0>0$, we multiply this inequality by ${\rm e}^{\nu(t-t_0)}$, we integrate over $(t_0,t)$ and we let $t\to\infty$ in order
to obtain the two inequalities in \eqref{Vnulimsup}.\end{proof}

Lemma \ref{energybound-PDE} should be read in the following way. Once the damping $k$ is fixed, one can choose $\nu(k)$ and then $\delta(k)$ such that $0<\nu\le \frac k2$,~ $0<\delta<\frac{k-\nu}2$ and
$$\frac{4\delta(k)(1-\sigma^2)\nu(k)(\lambda_1-P-\nu(k) k +\nu(k)^2)}{\lambda_1}=\alpha_k^2>0.$$
Then Lemma \ref{energybound-PDE} gives a constant $L(k)>0$ such that
$$
V_\nu(\infty):=\limsup_{t\to\infty}V_\nu(t)\le L(k)\|g\|_{0}^2\,
$$
as soon as $\alpha^2\le \alpha_k^2$.
To get the flavour of the conditions in Lemma \ref{energybound-PDE}, consider for instance, without aiming to optimize, $\nu=k/2$ and $\delta=k/8$ and assume $4(\lambda_1-P)>k^2$.
If $${\left(\frac\alpha{k}\right)}^2\le  \frac{1-\sigma^2}{16\lambda_1}\Big(4(\lambda_1-P)-k^2\Big),$$ then
$$
V_\nu(\infty)\le \frac{4}{k^2}\|g\|_{0}^2\, .
$$
When $k$ is large, one cannot expect a bound of the order $1/k^2$. This is easily seen from the case $\alpha=0$ for which the conclusion can be quantified in an almost optimal way, see \cite{bongazmor} and \cite{ghisigobbinoharaux,2012GasmiHaraux,fitouriharaux}.
In the case $g=0$, the next lemma can be proved arguing as for Lemma \ref{energybound-PDE}.

\begin{lemma}\label{energybound-PDE-g=0}
Assume that $g=0$, $0\le P<\lambda_1$ and $u$ solves \eqref{plate}.
\begin{enumerate}[(a)]
\item  When $k^2\le 2(\lambda_1-P)$ and
$$|\alpha|\le k\sqrt{\frac{1-\sigma^2}{2\lambda_1}\left(\lambda_1-P-\frac{k^2}4\right)},$$
we have
$\lim_{t\to\infty}V_{k/2}(t)= 0$.
\item When $k^2\ge2(\lambda_1-P)$ and
$$|\alpha|\le \sqrt{\frac{1-\sigma^2}{2\lambda_1}}\left({\lambda_1-P}\right),\qquad\nu=\frac{k}2-\frac12\sqrt{k^2-2(\lambda_1-P)},
$$
we have
$
\lim_{t\to\infty}V_\nu(t)= 0.
$
\end{enumerate}
\end{lemma}

Note that in the two situations (a) and (b) of Lemma \ref{energybound-PDE-g=0}, the condition \eqref{positiveVnu} holds; moreover,
the smallness of $|\alpha|$ is related to \eqref{boundalpha}.
Hence, arguing as in \cite{bongazmor}, it can be checked that the bound on $V_\nu(t)$ gives asymptotic bounds on the norms of the solution.

\begin{lemma}\label{lem:L2bound}
Assume that $0\le P<\lambda_1$, $k>0$ and $g\in {L^2}$.
Let $\nu$ and $V_\nu(\infty)$ be as in Lemma $\ref{energybound-PDE}$. Then, if $u$ is a solution of \eqref{plate}, we have the following estimates:
\begin{eqnarray*}
&\bullet \ \mbox{$L^2$-bound on $u$}\qquad
&\limsup_{t\to \infty}\|u(t)\|^2_{0}\le\frac{4V_\nu(\infty)}{\sqrt{(\lambda_1-P)^2+4SV_\nu(\infty)}+(\lambda_1-P)}=:\Psi\, ;\\
&\bullet \ \mbox{$L^2$-bound on $u_x$}\qquad
&\limsup_{t\to\infty}\|u_x(t)\|_{0}^2\le \frac{4V_\nu(\infty)+2\nu^2\Psi}{\sqrt{(\lambda_1-P)^2+2S(2V_\nu(\infty)+\nu^2\Psi)}+(\lambda_1-P)}\, ;\\
&\bullet \ \mbox{$H^2$-bound on $u$}\qquad
&\limsup_{t\to\infty}\|u(t)\|_{H^2_*}^2\le \frac{2\lambda_1}{\lambda_1-P}\left(V_\nu(\infty)+ \frac{\nu^2\, \Psi}{2}\right)\, ;
\end{eqnarray*}
\end{lemma}

All the bounds obtained so far also hold for the weak solutions of the {linear problem}
\begin{empheq}{align}
\label{perturb-eq}
\left\{
\begin{array}{rl}
w_{tt}+ k w_t + \Delta^2 w + Pw_{xx} - b w -\alpha w_y= h(\xi,t)  &\textrm{in }\Omega\times(0,T)\\
w = w_{xx}= 0 &\textrm{on }\{0,\pi\}\times[-\ell,\ell]\\
w_{yy}+\sigma w_{xx} = w_{yyy}+(2-\sigma)w_{xxy}= 0 &\textrm{on }[0,\pi]\times\{-\ell,\ell\}.
\end{array}
\right.
\end{empheq}
Indeed, one just assumes $S=0$ in \eqref{plate} and takes care of the additional zeroth order term $b w$. More precisely, we have
(with the bound \eqref{boundonalpha} as in Lemma \ref{energybound-PDE})

\begin{lemma}\label{boundsforlinear}
Let $h\in C^0(\R_+,L^2(\Omega))$. Assume that $0\le P<\lambda_1$, $k>0,$ $0<\nu\le \frac k2$, $0<\delta<\frac{k-\nu}2$,
\begin{equation}\label{boundonalpha}
\alpha^2\le  \frac{4\delta(1-\sigma^2)\nu(\lambda_1-P-b-\nu k +\nu^2)}{\lambda_1}.
\end{equation}
Then the exists $L_1=L_1(\lambda_1,P,k)>0$ and $L_2=L_2(\lambda_1,P,k)$ such that for any weak solution $w$ of \eqref{perturb-eq}, we have the estimates
\begin{align}
\label{utbound-lin} &\bullet \ \mbox{$L^2$ bound on $w_t$}
&\limsup_{t\to\infty}\|w_t(t)\|_{0}^2\le {L_1}\limsup_{t\to\infty}\|h(t)\|_{0}^2;  \\ 
\label{uH2bound-lin} &\bullet \ \mbox{$H^2$ bound on $w$}
&\limsup_{t\to\infty}\|w(t)\|_{H^2_*}^2\le{L_2}\limsup_{t\to\infty}\|h(t)\|_{0}^2.
\end{align}
\end{lemma}

Next, we introduce a second parameter in the Lyapunov-type function $V_{\nu}$ in \eqref{lyapunov}
\begin{align}\label{lyapunov2}
V_{\nu,k}(S_t(y)) := &~\mathcal E(t)+\nu\Big(\big(u_t(t),u(t)\big) +\frac{k}{2}\|u\|_{0}^2\Big),
\end{align}
where $S_t(y) := y(t)= (u(t),u_t(t))$ for $t \ge 0$ and $\nu$ is a  positive number to be specified below.

\begin{lemma}\label{energies*} There exists $\nu_0=\nu_0(k)>0$ such that, if $0< \nu \le \nu_0$, there are $c_0,c_1,c_2 >0$ so that
\begin{equation}\label{energybounds}
c_0E_+ - c_2 \le V_{\nu,k}(S_t(y)) \le c_1E_+ +c_2.
\end{equation}\end{lemma}
\begin{proof}We claim that there exist $c,C,$ and $M$, which may depend on $P$, $g$, and $S$,  {but not on the particular trajectory}, such that
\begin{equation}\label{energybound}
cE_+-M \le \mathcal E \le CE_++M
\end{equation}
Indeed, using Young and Poincar\'{e} inequalities, we obtain for all $\nu,\delta>0$:
\begin{align*} ||u_x||_0^2 \le\nu ||u_x||_0^4+\frac{1}{4\nu}\ ,\quad
{(g,u) \le\delta ||u||_0^2+\frac{1}{4\delta}||g||_0^2\le\nu || u||_{H^2_*}^2+\frac{1}{4C_P\nu}||g||_0^2}\end{align*}
and \eqref{energybound} follows.
Next, we observe that
\begin{align*}
 \nu|(u_t,u)|+ \frac{\nu k}{2} (u,u) \le \gamma_1 ||u_t||_0^2+\left[\frac{\nu^2}{4\gamma_1}+\frac{\nu k}{2}\right]||u||_0^2,
\end{align*}
so that if $\nu_0<\min\{1,2/k,2kc\}$, then
\begin{align*}
\nu|(u_t,u)|+ \frac{\nu k}{2} (u,u) \le \gamma_1 ||u_t||_0^2+\left[\frac{1}{k^2\gamma_1}+1\right]||u||_0^2\le \frac{c}{k^2}E_+
\end{align*}
once we have selected $\gamma_1$ and used Poincar\'{e} inequality. On the other hand, we have
$$\nu(u_t,u)+ \frac{\nu k}{2} (u,u)  \ge -\frac{1}{4\gamma}\|u_t\|_0^2+\left[\frac{\nu k}{2}-\frac{\gamma\nu^2}{2}\right]\|u\|_0^2
\ge -\frac{1}{2\gamma} E_++\left[\frac{\nu k}{2}-\frac{\gamma\nu^2}{2}\right]\|u\|_0^2\ge 
-\frac{\nu}{2k}E_+
$$
which proves the bound thanks to the above smallness assumption  $\nu < \nu_0$.
\end{proof}

\subsection{Construction of an absorbing ball}\label{ball}

The main focus of this section is to show that the dynamical system $(S_t,Y)$ is dissipative:

\begin{proposition}\label{absorbingball}
The dynamical system $(S_t,Y)$ corresponding to generalized solutions to \eqref{plate} has a uniformly absorbing set $\mathscr B$, as defined in Appendix \ref{appendix3}.
\end{proposition}
For both the proof of Proposition \ref{absorbingball} and for the validity of Theorem \ref{generic} we need the following statement.

\begin{lemma}\label{potentiallowerbound}
Let $a(\cdot,\cdot)$ be as in \eqref{biform}. For any $\eta\in(0,2]$ and $\gamma>0$ there exists $C_{\gamma,\eta}>0$ such that
\begin{align}\label{l:epsilon}
\|u\|^2_{2-\eta} \le&~ \gamma \Big(a(u,u)+\|u_x\|_{0}^4\Big) + C_{\gamma,\eta}\quad\forall u \in H_*^2.
\end{align}
\end{lemma}
\begin{proof}Let $0< \eta \le 2$ be given. We claim that the functional
$$\psi_A(w)=\|w\|_2^2+\|w_x\|_{0}^4-A\|w\|_{2-\eta}^2$$
is bounded from below on $H^2_*$ for every $A>0$. This claim implies that
$$\exists C(A,\eta)>0\quad\mbox{s.t.}\quad\|w\|_2^2+\|w_x\|_{0}^4-A\|w\|_{2-\eta}^2\ge - C(A,\eta)\qquad\forall w\in H^2_*.$$
Therefore, we have
$$\|w\|_{2-\eta}^2 \le \frac1A \big(\|w\|_2^2+\|w_x\|_{0}^4+C(A,\eta)\big)$$
and the lemma will follow, since the inequality holds for every $A>0$.\par
In order to prove the claim, suppose to the contrary that, for some $A,\eta>0$,
there is a sequence $\{w_n\}\subset H^2_*$ such that $\psi_A(w_n) \to -\infty$ as $n\to\infty$. Up to relabeling, we may assume that $\psi_A(w_n)<0$ for all $n$ and, from the definition of $\psi_A$, we have that
~$\|w_n\|^2_{2-\eta} \to \infty.$
Writing $v_n=w_n/\|w_n\|_{2-\eta}$, we have
 $$\psi_A(w_n)=\|w_n\|_{2-\eta}^2\Big(\|v_n\|_2^2+\|w_n\|_{2-\eta}^2{\|[v_n]_x\|_{0}^4}-A\Big) \to -\infty,$$
and therefore we infer that  $$\Upsilon_n:=\|v_n\|_2^2+\|w_n\|_{2-\eta}^2{\|[v_n]_x\|_{0}^4}$$
 is bounded.
In particular, $v_n$ is bounded in $H^2_*$ so that there exists a subsequence (still denoted by $v_n$) and a function $v\in H^2_*$ such that $v_n \rightharpoonup  v$.  By the compactness of the Sobolev embeddings, $v_n \to v$ in any lower Sobolev norm and clearly $v\ne 0$.
The boundedness of $\|w_n\|_{2-\eta}^2{\|[v_n]_x\|_{0}^4}$ and the fact that $\|w_n\|^2_{2-\eta}\to \infty$ imply that
$\|[v_n]_x\|_{0} \to 0$, and thus $v_x=0$ by compactness. This means that $v$ is a function of $y$ only and the boundary conditions in $H^2_*$ yield $v=0$,
which is a contradiction.
\end{proof}

The next step is to bound the derivative of the Lyapunov function $V_{\nu,k}$ introduced in \eqref{lyapunov2}.

\begin{lemma}\label{le:48}
Let $V_{\nu,k}$ be as in \eqref{lyapunov2}. For all $k> 0$, there exist $\nu(k,S) >0$ sufficiently small, and
$c(\nu,k,\alpha, S),~ C(\alpha, k, g,P)>0$ such that
\begin{equation}\label{goodneg}
\dfrac{d}{dt}V_{\nu,k}(
S_t(y))\le-cE_+(t)+C.
\end{equation}
\end{lemma}

\begin{proof}Suppose $y(t)=(u,u_t)$ is a smooth solution of \eqref{plate} in $\mathcal D(A) \times H^2_*$ (we can then extend by density to generalized solutions in the final estimate). Then
\begin{align}
\dfrac{dV_{\nu,k}}{dt} =&~ a(u,u_t)+(u_t,u_{tt})+S||u_x||_0^2(u_x,u_{xt})-P(u_x,u_{xt})-(g,u_t)+\nu(u_{tt},u)+\nu||u_t||_0^2+k\nu(u,u_t) \\
=&~(\nu-k)||u_t||^2_0-\nu a(u,u)-\nu S||u_x||^4_0+\alpha(u_y,u_t)+\nu\alpha (u_y,u)+\nu P||u_x||_0^2+\nu(g,u)
\end{align}
From here, we obtain:
\begin{align}
\dfrac{dV_{\nu,k}}{dt} \le &~(\nu-k)||u_t||^2_0-\nu a(u,u)-\nu S||u_x||^4_0+\alpha(u_y,u_t)+\nu\alpha (u_y,u)+\nu P||u_x||_0^2+\nu(g,u) \\
\le &~\big(\nu-\frac k2\big)||u_t||^2_0-\nu a(u,u)+\nu\big(\dfrac{\nu}{2}-S\big)||u_x||^4_0 +\frac{\alpha^2}{2}\big(\dfrac{1}{k}+1\big)||u_y||_0^2 +\dfrac{3\nu^2}{2}||u||_0^2\\
& +\dfrac{1}{2}(P^2+||g||_0^2)\\
\le &~ \big(\nu-\frac k2\big)||u_t||^2_0+\nu\big(\frac{3\nu}{2\lambda_1}-1\big) a(u,u)+\nu\big(\dfrac{\nu}{2}-S\big)||u_x||^4_0
+\frac{\alpha^2}{2}\big(\dfrac{1}{k}+1\big)||u||_1^2 +\dfrac{1}{2}(P^2+||g||_0^2).
\end{align}
From Lemma \ref{l:epsilon}, we infer that for any $\gamma>0$, there exists $C_\gamma>0$ such that
$$\|u\|^2_{1} \le \gamma \Big(a(u,u)+\|u_x\|_{0}^4\Big) + C_{\gamma}\quad\forall u \in H_*^2.$$
This yields 
\begin{align}
\dfrac{dV_{\nu,k}}{dt} \le &~ \big(\nu-\frac k2\big)||u_t||^2_0+\Big(\frac{\gamma\alpha^2}{2}\big(\dfrac{1}{k}+1\big)+\nu\big(\frac{3\nu}{2\lambda_1}-1\big)\Big) a(u,u)+\Big(\frac{\gamma\alpha^2}{2}\big(\dfrac{1}{k}+1\big)+\nu\big(\dfrac{\nu}{2}-S\big)\Big)||u_x||^4_0\\
& +\dfrac{1}{2}(P^2+||g||_0^2) + C_\gamma.
\end{align}
The conclusion follows by choosing $\gamma>0$ and $\nu<k/2$ such that 
$$\frac{\gamma\alpha^2}{2}\Big(\dfrac{1}{k}+1\Big)+\nu\Big(\frac{3\nu}{2\lambda_1}-1\Big)< 0\ \text{ and }\ \frac{\gamma\alpha^2}{2}\Big(\dfrac{1}{k}+1\Big)+\nu\Big(\dfrac{\nu}{2}-S\Big)<0.$$
\end{proof}

We are now ready to give the {proof of Proposition \ref{absorbingball}}.
From Lemma \ref{le:48} and the upper bound in \eqref{energybounds}, we have for some $\eta(\nu)>0$ and a $C$: \begin{equation}\label{gronish}
\dfrac{d}{dt}V_{\nu,k}(S_t(y)) +\eta V_{\nu,k}(S_t(y)) \le C,~~t>0.
\end{equation}
The estimate above in \eqref{gronish} implies (via an integrating factor) that
\begin{equation}\label{balll}
V_{\nu,k}(S_t(y)) \le V_{\nu,k}(y)e^{-\eta t}+\dfrac{C}{\eta}(1-e^{-\eta t}).
\end{equation}
Hence, the set
$$
\mathscr{B} := \left\{x \in Y:~V_{\nu,k}(x) \le 1+\dfrac{C}{\eta} \right\},
$$  is a bounded, forward invariant absorbing set, and $(S_t,Y)$ is ultimately dissipative.

\begin{remark} Unlike for exponential stability, as stated in Theorem \ref{th:main2}, $P$ and $\alpha$ may take any value, for
any fixed $S>0$, and the above absorbing ball is obtained. This illustrates the strength of Lemma \ref{potentiallowerbound}, namely the
ability for the nonlinear potential energy to control low frequencies. \end{remark}

\subsection{Further estimates and identities}

Let $f(u)=[P-S\|u_x\|_{0}^2]u_{xx}$. Consider the difference of two strong solutions $u^i$, $i=1,2$, to \eqref{plate}, satisfying:
\begin{equation}\label{difference}
\begin{cases}z_{tt}+\Delta ^2 z + k z_t+f(u^1)-f(u^2)=\alpha z_y,\\
z = z_{xx}= 0 &\textrm{on }\{0,\pi\}\times[-\ell,\ell]\\
z_{yy}+\sigma z_{xx} = z_{yyy}+(2-\sigma)z_{xxy}= 0 &\textrm{on }[0,\pi]\times\{-\ell,\ell\}\\
z(x,y, 0) = u^1_0(x,y)-u^2_0(x,y), \quad \quad z_t(x,y,0) = v^1_0(x,y)- v^2_0(x,y) &\textrm{in }\Omega\, .
 \end{cases}
\end{equation}
We take this equation with the notations:
\begin{equation}\label{Ez} z=u^1-u^2;~~ E_z(t) := \dfrac{1}{2}\Big\{a(z,z) + \|z_t(t)\|_{0}^2\Big\};~~\cF(z)=f(u^1)-f(u^2).\end{equation}

We utilize a decomposition of the term $\int_{\Omega}\cF(z)z_t$.
Results in the next statement follow from direct calculation and can be found in \cite{Memoires,gw,HLW} for the Woinowsky-Krieger type nonlinearity, though we consider the details below for our specific  hinged-free configuration. The calculations are done on smooth functions in $(u(t),u_t(t)) \in \mathcal D(A) \times H^2_*$ then extended by density below.
\begin{proposition}\label{nonest}
Let $u^i \in {B}_R(H^2_*)$, $i=1,2$. Then  we have:
\begin{equation}\label{f-est-lip}
\|f(u^1)-f(u^2)\|_{0}  \le 
C(R)\|u^1-u^2\|_{2}.
\end{equation}
In addition, for $u^1,u^2 \in
C^0(\R_+;H_*^2)\cap C^1(\R_+;L^2(\Omega))$, writing $z=u^1-u^2$, we have:
\begin{equation*}
\big( \cF(z),z_{t}\big) =\dfrac{1}{2}\dfrac{d}{dt}\Big[S\| u_x^1\|_0^2\| z_x\|_0^2-P\| z_x\|_0^2 \Big]+S||z_x||^2_0(u^1_{xx},u^1_t)
-S [\| u_x^1\|_0^2-\| u_x^2\|_0^2]( u_{xx}^2,z_t )
\end{equation*}
\end{proposition}

\begin{proof}
Letting $z=u^1-u^2$, and letting $N(u)=(P-S\|u_x\|_0^2)$, we note two facts immediately:
\begin{align*}
N(u^1)-N(u^2) = P z_{xx} -S\big[\| u_x^1\|_0^2 u_{xx}^1-\|u_x^2\|_0 u_{xx}^2\big]
=Pz_{xx} -S\big[\| u_x^1\|_0^2z_{xx}+(\|u_x^1\|_0^2-\|u_x^2\|_0^2) u_{xx}^2\big] \\[.2cm]
\left|~\|u_x^1\|_0^2-\|u_x^2\|_0^2~\right|=\Big\||u_x^1\|_0-\|u_x^2\|_0\Big|\left(\|u_x^1\|_0+\|u_x^2\|_0\right)
\le \left(\|u_x^1\|_0+\|u_x^2\|_0\right)\|u_x^1 -  u_x^2\|_0\le C(R)\|z\|_1,
\end{align*}
From here, note that
$$
\| \mathcal F(z) \|_{0} = \| N(u^1)u^1-N(u^2)u^2 \|_0
 \le P\| z_{xx}\|_0+S\| u_x^1\|_0^2\| z_{xx}\|_0+\| u_{xx}^2\|_0\big[\| u_x^1\|_0^2-\| u_x^2\|_0^2\big]
 \le C(R)
 \|z\|_{2},
$$
as desired. For the decomposition, we have:
\begin{align*}
( \mathcal F(z),z_t) =&~ P(  z_{xx}, z_t) -S( \| u_x^1\|_0^2 z_{xx},z_t) - S\left(  u_{xx}^2[\| u_x^1\|_0^2-\| u_x^2\|_0^2],z_t\right )\\
=&~ \dfrac{1}{2}\dfrac{d}{dt}\Big[S\| u_x^1\|_0^2\| z_x\|_0^2-P\| z_x\|_0^2 \Big]-\dfrac{S}{2}\| z_x\|_0^2 \dfrac{d}{dt}\| u_x^1\|_0^2
-S [\| u_x^1\|_0^2-\| u_x^2\|_0^2]( u_{xx}^2,z_t )\\
=&~ \dfrac{1}{2}\dfrac{d}{dt}\Big[S\| u_x^1\|_0^2\| z_x\|_0^2-P\| z_x\|_0^2 \Big]+S||z_x||^2_0(u^1_{xx},u^1_t)
-S [\| u_x^1\|_0^2-\| u_x^2\|_0^2]( u_{xx}^2,z_t ).  
\end{align*}
Above, we have integrated by parts.
\end{proof}

We note the following identities (again, obtained first on strong solutions, and then passing to the limit for generalized solutions) corresponding to \eqref{difference}. The first is the energy identity, and the second is reached via using the solution itself as a multiplier (equipartition type):

\begin{align}\label{stufff}
E_z(t) + k\int_s^t\|z_t\|_{0}^2 =&~E_z(0)-\int_s^t\big(\cF(z),z_t\big) +\alpha\int_s^t\big(z_y,z_t\big)\\[.2cm]
\int_s^t a(z,z) - \int_s^t \|z_t\|_{0}^2 =&~ \dfrac{k}{2}\|z\|_{0}^2\Big|_s^t+\alpha\int_s^t\big(z_y,z\big)-\int_s^t\big(\cF(z),z\big)
\end{align}

The following lemma is a special case of \cite[Lemma 8.3.1, p.398]{springer}. It is a standard estimate
utilizing \eqref{stufff} (with $k>0$) and the fact that $f \in Lip_{loc}\big(H_*^2, L^2(\Omega)\big)$.

\begin{lemma}\label{le:observbl}
Let $u^i \in C^0(\R_+;H_*^2)\cap C^1(\R_+;L^2(\Omega))$ solve \eqref{plate} on $\R_+$ for $i=1,2$.
Additionally, assume $(u^i(t),u^i_t(t)) \in B_R(Y)$ for all $t\in \R_+$. Then, for any $\eta\in(0,2]$, and any $T>0$:
\begin{align}\label{enest1}
T\Ez(T)+\int_0^T \Ez(\tau) d\tau \le& ~ a_0\Ez(0)+C(\eta,T,R)\sup_{\tau \in [0,T]}\|z\|^2_{2-\eta} \nonumber \\
&-a_1\int_0^T\int_s^T \big( \cF(z),z_t\big) d\tau ds  -a_2\int_0^T \big( \cF(z),z_t\big) ds
\end{align}
hold with $a_i>0$ independent of $T$ and $R$.
\end{lemma}

\section{Quasi-stability and attractors: proof of Theorem \ref{th:main1}}\label{secattract}

In this section we construct the global compact attractor for the dynamics $(S_t,Y)$ using quasi-stability theory \cite{quasi}. A quasi-stable dynamical system is one where {the difference of two trajectories can be decomposed} into uniformly stable and compact parts, with controlled scaling of powers. Using this theory, it is also possible to obtain, almost immediately, that the attractor is smooth, with finite fractal dimension, and that there exists a generalized fractal exponential attractor. See Appendix \ref{appendix3} for relevant definitions and theorems.\par
We adopt the tack  of showing the quasi-stability estimate \eqref{specquasi*} on the absorbing ball given in Proposition \ref{absorbingball}.
Obtaining quasi-stability on $\mathscr B$ will follow directly from the observability inequality \eqref{enest1} and the nonlinear decomposition of Proposition \ref{nonest}. In fact, the proof below demonstrates the quasi-stability estimate on any bounded, forward invariant set.

\begin{lemma} \label{quasistep} Let $k,S>0$ and $\alpha,P \in \mathbb R$. The dynamical system $(S_t,Y)$ corresponding to generalized solutions to \eqref{plate} is {quasi-stable} on any bounded, forward invariant set. In particular, $(S_t,Y)$ is quasi-stable on the absorbing ball $\mathscr B$ given in Lemma \ref{ball}.\end{lemma}

\begin{proof}
Let $z=u^1-u^2$ and consider the decomposition as in Proposition \ref{nonest}:
\begin{align*}
\big( \mathcal F(z),z_t\big) =\dfrac{1}{2}\dfrac{d}{dt}\Big[S\| u_x^1\|_0^2\| z_x\|_0^2-P\| z_x\|_0^2 \Big]+{S}\| z_x\|_0^2 \big(  u_{xx}^1,u^1_t\big)-S [\| u_x^1\|_0^2-\| u_x^2\|_0^2]\big( u_{xx}^2,z_t \big).
\end{align*}
Now on any bounded, forward-invariant ball $B_R(Y)$ ($R$ is the radius)\begin{equation*}
\|u^1(t)\|_{2}+\|u^1_{t}(t)\|_{0}+\|u^2(t)\|_{2}+\|u^2_{t}(t)\|_{0}\leq
C(R),~~t>0,
\end{equation*}
and the Lipschitz bound \eqref{f-est-lip}, it follows immediately from the Cauchy-Schwarz, triangle, and Young inequalities, that, for  $0<\eta<1/2$:
\begin{align}\label{mmmbop}
\Big|\int_s^t \big( \cF(z),z_t\big) d\tau\Big| \le &~C(\eta,R,\gamma)\sup_{\tau \in [s,t]} \|z\|^2_{2-\eta}+ \gamma\int_s^tE_{z}(t)d\tau,~~\forall\gamma>0,
\end{align} provided
 $u^i(\tau) \in \mathscr{B}_R(H^2_*)$ for all $\tau\in [s,t]$. In particular, this bound holds on the invariant, absorbing ball
 $\mathscr B$ from Proposition \ref{absorbingball}.

By \eqref{enest1}--\eqref{mmmbop}, and taking $T$ sufficiently large, we infer
from the observability inequality that:
\begin{equation*}
E_{z}(T) \leq c E_z(0)+C(R,T,k,\eta)\underset{\tau \in \lbrack0,T]}{\sup }\|z(\tau )\|_{2-\eta }^2
\end{equation*}
with $c<1$.
By the standard iteration via the semigroup property, we conclude that
\begin{equation}\label{qs*}
\big|\big|(z(t),z_t(t))\big|\big|_{Y}^2 \le C(\sigma, R)e^{-\sigma t}\big|\big|(z(0),z_t(0))\big|\big|_{Y}^2+C( R,k,\eta) \sup_{\tau \in [0,t]} \|z(\tau)\|_{2-\eta}^2,
\end{equation}
 and thus $(S_t,Y)$ is quasi-stable on $ B_R(Y)$, as desired.
 \end{proof}

On the strength of Theorem \ref{doy}, applied with $B =\mathscr B \subset Y$, we
deduce the existence of a compact global attractor from the quasi-stability property of $(S_t,\mathscr B)$. In addition, since $\mathbf A \subseteq \mathscr B$, Theorem \ref{dimsmooth} guarantees $\mathbf{A}$ has finite fractal dimension and that $$\|u_{tt}(t)\|_0^2+\|u_t(t)\|_2^2 \le C(\mathbf A)~\text{ for all } t \ge0.$$
Since $u_{t}\in H^{2}(\Omega )\subset C(\overline \Omega)$, elliptic regularity for the free-hinged rectangular plate \cite{2015ferreroDCDSA,bookgaz}
\begin{equation}\label{forelliptic}
\Delta ^{2}u=g-u_{tt}-ku_t-f(u)+\alpha u_y \in L^2(\Omega)
\end{equation}
 gives immediately that $u \in H^4(\Omega)\cap H_*^2$, with the corresponding bound (in that topology) coming from the uniform-in-time bound on the RHS of \eqref{forelliptic} and the equivalence $||\cdot||_2 \approx ||\cdot||_{H^2_*}$.
Thus, we conclude the regularity of the attractor $\bA \subset Y$ as in Theorem \ref{th:main1}.\par
With the quasi-stability estimate established on the absorbing ball $\mathscr B$, we need only establish the H\"{o}lder continuity in time of $S_t$ in {some} weaker space $\widetilde Y$ to obtain a generalized fractal exponential attractor. This is done through lifting via the operator $A^{-1/2}$; recall that $Au=\Delta^2u$ on $\mathcal D(A)$ as in \eqref{Adef}.{
Via the standard construction \cite{springer,LT}, for $\phi \in L^2(\Omega)$, we obtain $ A^{-1/2}\phi \in H_*^2 = \mathcal D(A^{1/2})$.}
{
We may restrict our attention to the absorbing ball $\mathscr B$ (for $t>t(y(0))$): $\|y(t)\|_{Y} \le C(\mathscr B)$. In particular, for any $y(t)=(u(t),u_t(t))$, with $t$ sufficiently large,  we have global-in-time bounds: 
\begin{align}
\|u(t)\|_{H^2_*} \le C(\mathscr B),~ &~~\|u_t(t)\|_0 \le C(\mathscr B) \implies E_+(t) \le \frac{1}{c_0}\Big[V_{\nu,k}(S_t(y))+c\Big] \le C(\mathscr B).
\end{align}
And thus we have from the equation \eqref{plate}  (on strong form) \begin{align}  A^{-1/2} u_{tt}=&~ A^{1/2}u + A^{-1/2}\Big[g+\alpha u_y -k  u_t - f(u)\Big].
 \end{align}
We can estimate by duality for $\phi\in L^2(\Omega)$: $$( A^{-1/2} u_y, \phi )_{L^2(\Omega)}  = (u_y , A^{-1/2} \phi )_{L^2(\Omega)} = -(u, \partial_y A^{-1/2} \phi)_{L^2(\Omega)}+(u,A^{-1/2}\phi)_{L^2(\{y=\pm \ell \})}$$ via integration by parts in $y$. Since $\phi \in L^2(\Omega)$ gives  $\partial_y A^{-1/2} \phi \in H_*^1(\Omega) \equiv \{ u \in H^1(\Omega)~:~u(0,y)=u(\pi,y)=0\}$, and making use of the trace theorem's estimate for the boundary term,  we have:
$$|(A^{-1/2} u_y, \phi) | \leq C||u||_0||\phi||_{0}+||u||_{H^{1/2+\epsilon}(\Omega)}||A^{-1/4}\phi||_{0} \le C||u||_{H^{1/2+\epsilon}(\Omega)}||\phi||_0.$$
The Riesz Representation Theorem then yields
 $$||A^{-1/2}u_y||_{0} \le C||u||_{H^{1/2+\epsilon}(\Omega)},$$
from which it follows that
 \begin{align*} \| A^{-1/2}u_{tt}\|_{0} 
 \le &~C(\alpha)\|u\|_{H_*^2}+C(k)\|u_t\|_{0}+C\|g\|_{0}\le C(\alpha,g,\mathscr B).
 \end{align*}}
From here, we note
$u_t(t)-u_t(s) =\int_s^tu_{tt}(\tau)d\tau,$ and thus
\begin{equation}
\|u_t(t)-u_t(s)\|_{[H^{2}_*]'} \le C\| A^{-1/2} [u_t(t)-u_t(s)]\|_{0} \le C\int_s^t \| A^{-1/2}u_{tt}(\tau)\|_0d\tau  \le C(\alpha,k,\mathscr B)|t-s|,
\end{equation}
which extends to generalized solutions as before. Lastly, we note 
\begin{align}
\|u(t)-u(s)\|_0 \le &~ \int_s^t \|u_t(\tau)\|_0 d\tau  \le \Big(\sup_{t \ge 0}\|u_t\|_0\Big)|t-s|\le C(\mathscr B)|t-s|
\end{align}
From the above estimates, we see that
$$\|S_t(y)-S_s(y)\|_{\widetilde Y} \le \mathcal C|t-s|,~~\widetilde{Y}=L^2(\Omega)\times [H_*^2]'$$
and thus we note that $(S_t,Y)$ is uniformly-in-time Lipschitz continuous on $\mathscr B$ in the sense of  $\widetilde{Y}$.

\section{Convergence to equilibrium I: proof of Theorem \ref{th:main2}}\label{pointstab}

A preliminary step is to notice that Theorem \ref{generic} ensures the existence of a solution to the stationary equation
\eqref{stationaryplate}. Then we prove the three statements in Theorem \ref{th:main2} in an order different than they are stated. First, we show that, under smallness assumptions, the unique stationary solution is the trivial one: to this end, we need an a priori
bound which depends only on $g$. The two other statements are proved under the same principle, namely that the very same smallness
conditions enable us to prove exponential stabilization of any difference of solutions, which, in turn, implies uniqueness
of stationary solutions.\par\smallskip

Multiplying \eqref{stationaryplate} by the solution itself and integrating by parts, we obtain
\begin{eqnarray}
\frac{\lambda_1-P}{\lambda_1}\|u\|_{H^2_*}^2 &<& \|u\|_{H^2_*}^2-P\|u_x\|_{0}^2+S\|u_x\|_{0}^4 =\alpha \int_\Omega u_yu\, d\xi + \int_\Omega gu\, d\xi \\
& \le &|\alpha|\cdot\|u_y\|_{0}\|u\|_{0} + \|g\|_{0}\|u\|_{0}\le\frac{|\alpha|}{\sqrt{2\lambda_1(1-\sigma^2)}} \|u\|_{H^2_*}^2+\frac1{\lambda_1} \|g\|_{0}\|u\|_{H^2_*}\, ,
 \end{eqnarray}
where we used the H\"older inequality and the embedding inequalities \eqref{embedding} and \eqref{embedding-H^2_*-u_y}. Therefore, if
\begin{equation}\label{lowerbeta}
|\alpha|<\frac{\lambda_1-P}{\sqrt\lambda_1} \sqrt{2(1-\sigma^2)}\, ,
\end{equation}
we deduce that
$$\|u\|_{H^2_*}\le \frac{\sqrt{2(1-\sigma^2)}}{(\lambda_1-P)\sqrt{2(1-\sigma^2)}-{|\alpha|\sqrt\lambda_1}} \|g\|_{0},$$
an a priori bound for stationary solutions. In particular, this shows that if $g=0$ and \eqref{lowerbeta} holds, then the
unique stationary solution is $u_g=0$, thereby proving the last statement in Theorem \ref{th:main2}.\par
For the remaining statements (when $g\ne 0$), arguing as in \cite[Section 7]{bongazmor}, one deduces the following result from Lemma \ref{boundsforlinear}.

\begin{lemma}\label{lemme-naze}
There exists $g_0=g_0( k,S,P,\lambda_1)>0$ and $\alpha_k=\alpha_k( k,\sigma,P,\lambda_1)>0$ such that if
\begin{equation}\label{g0}
\|g\|_{0}<g_0 \text{ and } | \alpha | <\alpha_k,
\end{equation}
then there exists $\eta>0$ such that, for any two solutions $u$ and $v$ of \eqref{plate}, we have
$$\lim_{t\to\infty} {\rm e}^{\eta t}\left(\|u_t(t)-v_t(t)\|_{0}+\|u(t)-v(t)\|_{H^2_*}\right) =0.$$
\end{lemma}
\begin{proof}Take $0<\nu\le \frac k2$ such that $\lambda_1-P>\nu(k-\nu)$, take $0<\delta<\frac{k-\nu}2$, and put
$$\alpha_k^2 = \frac{4\delta(1-\sigma^2)\nu(\lambda_1-P-\nu k +\nu^2)}{\lambda_1}>0.$$
If $\alpha^2<\alpha_k^2$ (as in \eqref{boundonalpha}), then there exists $\eta>0$ such that
$$0<\alpha^2=\frac{4\delta(1-\sigma^2)\nu(\lambda_1-P-\eta-\nu k +\nu^2)}{\lambda_1}.$$
If $u$ and $v$ are two solutions of \eqref{plate}, then $w=(u-v){\rm e}^{\eta t}$ is such that
$$
\langle w_{tt},\varphi \rangle  + ( k-2\eta) (w_t,\varphi) + a(w,\varphi)-P(w_x,\varphi_x)-\eta( k-\eta) (w,\varphi) - \alpha(w_y,\varphi) = (h{\rm e}^{\eta t},\varphi)
$$
for all $t\in[0,T]$ and all $\varphi\in H^{2}_*(\Omega)$, where
\begin{eqnarray*}
h(\xi,t) {\rm e}^{\eta t} &=& S{\rm e}^{\eta t}\left( u_{xx}(\xi,t)\|u_x\|_0^2- v_{xx}(\xi,t)\|v_x\|_0^2\right)\\
&=& S\Big(u_{xx}(\xi,t){\rm e}^{\eta t} (\|u_x(t)\|_0^2-\|v_x(t)\|_0^2)+w_{xx}(\xi,t)\|v_x(t)\|_0^2\Big).
\end{eqnarray*}
Therefore, we have
$$\|h(t) {\rm e}^{\eta t}\|_{0}\le S\big(\|u_{xx}(t)\|_{0}\|w_x(t)\|_{0}\|u_x(t)+v_x(t)\|_{0}+\|w_{xx}(t)\|_{0}\|v_x(t)\|_{0}^2\big)$$
so that, by combining \eqref{embedding} with Lemma \ref{lem:L2bound}, we deduce that there exists $C(\|g\|_0)>0$ such that
\begin{equation}\label{fundamental}
\limsup_{t\to\infty}\|h(t) {\rm e}^{\eta t}\|_{0}^2\le C(\|g\|_0) \limsup_{t\to\infty}\|w(t)\|_{H^2_*}^2
\end{equation}
and, for a family of varying $g\in L^2(\Omega)$, we have
~~$
C(\|g\|_0)\to 0\quad\mbox{if}\quad \|g\|_0\to 0.
$
Therefore, if $L_2$ is as in \eqref{uH2bound-lin}, and $g_0$ and $|\alpha|$ are sufficiently small (to satisfy \eqref{g0}, which yields both $\|g\|_0+|\alpha|<C(k)$ for some $C(k)$ and \eqref{boundonalpha}), we have that $L_2 C(\|g\|_0)<1$.
Taking into account the $H^2$-estimate \eqref{uH2bound-lin} for the linear equation \eq{perturb-eq} and using \eqref{fundamental}, we get
$$\limsup_{t\to\infty}\|w(t)\|_{H^2_*}^2\le L_2\limsup_{t\to\infty}\|h(t) {\rm e}^{\eta t}\|_{0}^2\le L_2 C(\|g\|_0) \limsup_{t\to\infty}\|w(t)\|_{H^2_*}^2\le\limsup_{t\to\infty}\|w(t)\|_{H^2_*}^2,$$
with strict inequality if the limsup differs from $0$. Therefore, we necessarily have $\|w(t)\|_{H^2_*}\to0$ as $t\to\infty$.
By undoing the change of variables, this proves that
$$\lim_{t\to\infty} {\rm e}^{\eta t}\|u(t)-v(t)\|_{H^2_*}=0.$$
By using \eqref{utbound-lin} we may proceed similarly to obtain
$$\lim_{t\to\infty} {\rm e}^{\eta t}\|u_t(t)-v_t(t)\|_{0}=0,$$
which concludes the proof.\end{proof}

The first two statements in Theorem \ref{th:main2} are straightforward consequences of Lemma \ref{lemme-naze}. First, by contradiction,
if there exist two stationary solutions $u_g^1$ and $u_g^2$, Lemma \ref{lemme-naze} states that
$$\lim_{t\to\infty} {\rm e}^{\eta t}\|u_g^2-u_g^1\|_{H^2_*}=0,$$
proving that $u_g^2=u_g^1$. With the uniqueness of $u_g$ at hand, we use Lemma \ref{lemme-naze} with a general solution $u=u(t)$ and
$v(t)\equiv u_g$ so that we obtain $$\lim_{t\to\infty} {\rm e}^{\eta t}\|u(t)-u_g\|_Y=
\lim_{t\to\infty} {\rm e}^{\eta t}\left(\|u_t(t)\|_0+\|u(t)-u_g\|_{H^2_*}\right) =0,$$
showing the uniform exponential decay of any solution $u=u(t)$ to $(u_g,0)$ in $Y$ as $t\to\infty$. This also completes the
proof of Theorem \ref{th:main2}

\section{Convergence to equilibrium II: proof of  Theorem \ref{th:main2a}}

For this section, recall that $W$ is the stationary set of weak solutions with properties given in Theorem \ref{generic}.
The proof of Theorem \ref{th:main2a} below depends on the conditions on $g$. In the first part, we take $g\equiv0$ and {assume} that $0 \in W$ is the unique stationary solution. Later, we modify that proof to obtain the result when $g \neq 0$ and $W$ may have multiple equilibria, so long as they are isolated and hyperbolic (the result for $u_g \in W$ the unique, hyperbolic stationary solution is included). \vskip.2cm

\noindent {\bf Part I -- Exponential convergence to zero for $g \equiv 0$ and $W=\{0\}$}.
\vskip.3cm
\noindent [\textbf{Step 1}] For any generalized solution to \eqref{plate} $(u(t),u_t(t))$ corresponding to the dynamical system $(S_t,Y)$, the following  {\it energy balance} is satisfied (see \eqref{energyrelation}):
\begin{equation}\label{balancehere}\mathcal{E}(t) + k \int_0^t \|u_t(s)\|_0^2ds = \mathcal{E}(0) + \alpha \int_0^t\big(u_y(s),u_t(s)\big)ds.\end{equation}
In view of \eqref{energies} (when $g=0$), the energy $\mathcal{E}$ is topologically equivalent to $E(t)$ by coercivity
($P<\lambda_1$); namely, there are $c,C>0$ such that:

\begin{equation}\label{topequiv} c E(t) \leq \mathcal{E}(t) \leq C [E(t) + E(t)^2 ].\end{equation}

\noindent [\textbf{Step 2}] Restricting our attention to the absorbing ball (for $t$ sufficiently large), we may invoke the observability estimate \eqref{enest1} on the difference of two trajectories on the absorbing ball, $z=u^1-u^2$. Coupling this with \eqref{mmmbop}, we obtain directly for $E_z(\tau) = \frac{1}{2}\big[a(z(\tau),z(\tau))+||z_t||_0^2\big]$:
\begin{equation}\label{obsbefore}
TE_z(T) + \int_0^TE_z(\tau) \le cE_z(0)+C(T,R)\sup_{\tau \in [0,T]}\|z(\tau)\|^2_1,
\end{equation}
where $C(R,T)>0$ depending on the radius of the absorbing ball $R$ and  $c>0$ is a generic constant.
Choosing $u^2=0$ (hence restricting $z=u^1=u$) and collecting these estimates, we obtain:
$$TE(T)+\cE(T)+\int_0^TE(t)dt \le cE(0)+\cE(0)+C(\alpha,T,R)\sup_{ t \in [0,T]} ||u(t)||_1^2.$$

From the energy balance \eqref{balancehere} and \eqref{topequiv}, we can directly estimate
\begin{equation}
cE(0) \le \cE(0) \le C(\alpha,T)\sup_{t \in [0,T]}||u(t)||_1^2+C(R)E(T)+k\int_0^T||u_t(t)||_0^2dt.\end{equation}
Fixing $T$ sufficiently large, we obtain the following observability estimate on a single trajectory:
\begin{equation} \label{obs}
\mathcal{E}(T) +\int_0^T E(t) dt \leq c \left [{k}  \int_0^T \|u_t\|_0^2 d\tau\right]+ k [lot(u)]
\quad\mbox{where}\quad\ds  {lot(u)  = C(T) \sup_{t\in[0,T]} ||u(t)||^2_1},
\end{equation}
and the quantity $\alpha^2k^{-1}$ is taken sufficiently small.

\noindent [\textbf{Step 3}] (Compactness--Uniqueness) Our aim in this step is to show the estimate
\begin{equation}\label{CU}
lot(u) \leq C(R)   \int_0^T \|u_t\|_0^2d\tau,
\end{equation}
for any generalized solution to \eqref{plate}, which will provide a true observability-type estimate.
This is a standard proof by contradiction. Assume the inequality \eqref{CU} is violated. Then,  there necessarily exists a sequence of generalized solutions, $\{(u_n(t),u_{n,t}(t))\}$ such that for all $n$, $$E_n({0}) \le M ~~\text{ with }~~ E_n(t) = \dfrac{1}{2}\Big[a\big(u_n(t),u_n(t)\big)+||u_{n,t}(t)||_0^2\Big],$$  and having the property that
\begin{equation}\label{cont}
 \dfrac{lot(u_n)}{\ds \int_0^T ||u_{n,t}||^2_0}   \rightarrow \infty,~~\text{ as }~~n \to \infty.
 \end{equation}

It is clear, for instance from \eqref{obsbefore}, that we have
 that $E_n(t) \leq C(M),~ t \in [0,T]$.  Hence $u_n$ has a weak limit  $u \in L^2(0,T; H^2(\Omega)) \cap H^1(0,T; L^2(\Omega))$. By the Aubin-Lions compactness criterion,
$$lot(u_n) \rightarrow lot(u).$$ Now, let us first assume that $ u \ne 0 $, so that $lot(u) \ne 0$.  The contradiction  hypothesis in \eqref{cont} implies that we must have
$$\int_0^T ||u_{n,t}||_0^2 dt \rightarrow 0.$$
It is also clear from boundedness of the energy on $[0,T]$ that
$$
u_n \rightharpoonup_* ~u\  \text{ in } \ L^{\infty}(0,T; H^2(\Omega) )\, ,\qquad
u_{n,t} \rightharpoonup_* ~u_t  \ \text{ in } \ L^{\infty}(0,T; L^2(\Omega))
$$
on appropriate subsequences denoted by the same index $n$.
On the other hand, $u_{n,t} \rightarrow 0 $ in $ L^2(0,T; L^2( \Omega))$. We consider the weak form of the plate equation \eqref{weakform} evaluated on solutions $(u_n,u_{n,t})$ and  pass to the {weak} limit. Limit passage on the linear terms is immediate, while the nonlinear term
$||u_{x,n}||_0^2u_{xx,n}$, being bounded in $L^2(0,T;L^2(\Omega))$, has a weak limit, $Q$. As it is a product of weakly convergent sequence  $u_{xx,n}$ in $L^2$ with a strongly convergent $u_{x,n} $ in $L^2$, it converges weakly in $L^1((0,T) \times \Omega)$ as a product to $||u_x||^2_0u_{xx}$, allowing us to identify $Q=||u_x||_0^2u_{xx}$. Hence, we may pass to the limit on a full nonlinear equation yielding the  limiting equation
$$\Delta^2u + [P - S ||u_x||_0^2]u_{xx} = \alpha u_y,$$
which $u \in H^2_*$  satisfies weakly. From the standing hypothesis that {no nontrivial weak steady states exist}, we infer that $u\equiv 0$, which contradicts our assumption (in this case) that $ u \ne 0$.

Next, let us consider the case when the limit point $u =0$, so that $lot(u_n) \rightarrow  lot(u) =0 $. {We may normalize by considering $v_n \equiv \dfrac{u_n}{lot(u_n)^{1/2}}$, then clearly  $lot(v_n)  \equiv 1$} and
$$ \frac{1}{\int_0^T ||v_{n,t}||^2_0}   \rightarrow \infty~~~\text{ or }~
\int_0^T ||v_{n,t}||^2_0 dt \rightarrow 0 $$
From the observability  inequality \eqref{obs}  renormalized by $lot(u_n)$, we also have
\begin{equation}\label{obs1}
\mathcal{E}_n(T) +\int_0^T E_n (t) dt \leq c   k \int_0^T ||v_{n,t} ||_0^2dt +k[lot(v_n)].\end{equation}
where  $\cE_n$ is $\cE$ evaluated on $(v_n,v_{n,t})$.
Since $\mathcal{E}_{n} \geq c E_{n}$ as in \eqref{topequiv},
$E_{n} (t) \in L^{\infty}(0,T)$, hence
$$
v_n \rightharpoonup_* ~v~\text{ in }~\ L^{\infty}(0,T; H^2(\Omega))\, ,\qquad
v_{n,t} \rightharpoonup_*~ v_t \text{ in } ~\ L^{\infty}(0,T; L_2(\Omega)),
$$
and $v_n$ satisfies:
$$ v_{n,tt} + k v_{n,t} +\Delta^2v_n + [P - S ||u_{n,x}||_0^2]v_{n,xx} =  \alpha v_{n,y}.$$
Since $u_n \rightarrow 0$ in $L^2(0,T; H^1(\Omega) $  and $v_{t,n} \rightarrow 0$ in $L^2( (0,T)\times \Omega)$ we can pass with the adapted weak form \eqref{weakform} to obtain the following equation for the weak limit $v \in H^2_*$:
$$ \Delta^2v + P v_{xx} =  \alpha v_{y}.$$
By the assumption of hyperbolicity of the zero equilibrium  (a sufficient condition being the smallness of $|\alpha|$), we obtain that $ v\equiv 0 $. This contradicts $\ds lot(v)= \lim_n lot(v_n) \equiv 1 \ne 0,$ where the latter limit  again follows from the compactness of  $lot(v)$ with respect to  energy $E(t)$.

Hence, in both cases, the estimate \eqref{CU} holds, which will be used in the next step.

\noindent [\textbf{Step 4}] Combining Steps 2 and 3, we have:
$$\mathcal{E}(T) +\int_0^T E(t) dt \leq C (R) \big[k  \int_0^T \|u_t(t)\|_0^2dt].$$

\noindent [\textbf{Step 5}] Directly from the energy balance and Young's inequality, we have for all $t$:
$$\mathcal{E}(t) + k \int_0^t \|u_t(s)\|_0^2ds\leq
\mathcal{E}(0) +\frac{1}{2k} \alpha^2\int_0^t  \|u_y(s)\|_0^2ds + \frac{k}{2}  \int_0^t\|u_t(s)\|_0^2ds.$$
This gives
$$\frac{ k}{2}  \int_0^t \|u_t(s)\|_0^2ds \leq \mathcal{E}(0)-\cE(t) +c\frac{\alpha^2}{k} \int_0^t  E(s)ds,$$
and, from Step 4, we obtain
\begin{equation}\label{ddd}\mathcal{E}(T) +\int_0^T E(t) dt \leq C k   \int_0^T \|u_t\|_0^2  \leq  C [\mathcal{E}(0) -\mathcal{E}(T)]  + c \frac{ \alpha^2}{k}  \int_0^T E(t).\end{equation}
Thus, there exists a number $q>0$ (depending only on $\lambda_1$ and $R$) so that if {$\alpha^2k^{-1}<q$},    the last term in \eqref{ddd} is absorbed by the integral of energy on the LHS:
$$\mathcal{E}(T) +\int_0^T E(t) dt  \leq  C [\mathcal{E}(0) -\mathcal{E}(T)].$$

This yields the traditional hyperbolic-type stabilizability estimate on $[0,T]$
$$\mathcal{E}(T) \leq \frac{C}{C+1} \mathcal{E}(0),$$ where $C>0$. Since the dynamics corresponding to $(S_t,Y)$ (and its restriction to the absorbing ball $\mathscr B \subset Y$) are autonomous and $T$ measures only the length of the time interval considered, we obtain exponential decay through the semigroup property and iteration.

The proof of Theorem \ref{th:main2}  (in the case when $g =0$ and stationary problem has 0 as the unique hyperbolic equilibrium) is concluded on the strength of the existence of the absorbing ball Theorem \ref{th:main1}.
\vskip.2cm
\noindent {\bf Part II -- General exponential decay for $g \ne 0 $}.

\noindent In this case, we consider a trajectory converging strongly to an isolated, hyperbolic equilibrium, as in the hypotheses. As such, consider
a trajectory $S_t(y_0) = (u(t), u_t(t)) \rightarrow
(e, 0)$ in $Y$ as $t \rightarrow \infty$, with $e \in W$ (i.e., $(e,0)$ is stationary point of $S_t$) and assume $e$ has a neighborhood in $H^2_*$ so that it is the unique element of $W$ in that neighborhood. (Remark \ref{R} below addresses the case $W=\{e\}$.)

\noindent {\bf [Step 1]}
Let us introduce  $z \equiv u - e$, yielding the trajectory $Z(t) \equiv (z(t),z_t(t)) \rightarrow 0$ as $t \rightarrow \infty$ strongly in $Y=H_*^2 \times L^2(\Omega)$. Let $\epsilon>0$. Since the $e$ is isolated,  there exists $ T_0(e,\epsilon) > 0$  so that
\begin{equation}\label{escape}
\int_{T-1}^T E_z(t) dt \leq \epsilon,  ~ \forall  ~T > T_0,
\end{equation}
The variable $z$ satisfies  the following equation weakly
\begin{equation}\label{zeq}
z_{tt} + k z_t + \Delta^2 z + Pz_{xx} - S ||u_x||_0^2 z_{xx} + S[ ||e_x||_0^2 -||u_x||_0^2 ]e_{xx} = \alpha z_y
\end{equation}
with the boundary conditions associated to $H^2_*$.
We shall show that $z $ converges exponentially to zero.
The key to the argument will be the functional
 $$\Phi(z) \equiv  \frac{S}{4} [||u_x||_0^4
-||e_x||_0^4] - S ||e_{x}||_0^2 \big(e_{x} ,z_x\big) = \frac{S}{4} [||e_x + z_x||_0^4
-||e_x||_0^4] - S ||e_{x}||_0^2 \big(e_{x} ,z_x\big).$$
It can be  verified  directly that
$$ \frac{d}{dt} \Phi(z(t) ) =S\Big(\big[||e_{x}||_0^2 -||u_x||_0^2\big]e_{xx} - ||u_x||_0^2 z_{xx}   , z_t\Big).$$
Now, let us define a Lyapunov function $$V(t)=V_e(t) \equiv E_z(t)-\dfrac{P}{2} ||z_x||_0^2 + \Phi(z(t)).$$
With the calculation of $\dfrac{d}{dt}\Phi(z(t))$ above and the equation \eqref{zeq},
we obtain the  identity:
\begin{equation}\label{Vbalance}
V(t) + k \int_s^t ||z_t||_0^2 d\tau = V(s) +  \alpha \int_s^t  (z_y, z_t )_{\Omega} d\tau
\end{equation}
Since $z \rightarrow 0 $ in the energy space when $t \rightarrow \infty $, the structure of $\Phi(z(t))$ clearly  has
$\Phi(z(t)) \geq 0$ and $V(t) \geq 0 $ for $t > T_0(e,\epsilon)$ (as in \eqref{escape}).
Additionally,
$$
V(t) \leq C(R) E_z(t) \, ,\qquad|E_z(t) - V(t)|\leq \rho ||z(t)|| ^2_{H^2_*} +  C(\rho, R) ||z(t)||_0^2,
$$
where $\rho$ can be taken arbitrarily small, and, as before, $R$ indicates the radius of the  absorbing ball.

\noindent {\bf [Step 2].} Proceeding as in {\bf Part I} of this section, we adapt the observability inequality:
\begin{equation}\label{obs2}
V(T)   +\int_s^T E_{z} (t) dt \leq k \left[C  \int_s^T \|z_{t} \|^2_0d\tau\right] +  k[lot(z)],~~0 <s <T
\end{equation} {
where, in this case, $$ lot(z) = \frac{\alpha^2}{k} \int_s^T ||z_y||^2_0 d \tau + \sup_{t\in [0,T] } ||z(t)||_1^2,$$
and again, $\alpha^2k^{-1}$ is taken sufficiently small.
As before, we eliminate the  lower order term ~$lot(z)$.

\noindent {\bf [Step 3.]} We state as a lemma the $lot(z)$ estimate.
\begin{lemma}\label{thisone}
Let z be a generalized  solution of \eqref{zeq} and such that
$\ds \sup_{t\in[0,T]}  E(z(t) ) \leq R^2 $.  Then , there exists $\epsilon_0 > 0 $ (as in \eqref{escape}) such that for $\epsilon < \epsilon_0 $
$$lot(z) \leq  \left[C( R,T_0,\epsilon)   \int_s^T ||z_t||_0^2 d\tau \right]$$
Here $ T_0 = T(S_t(y_0),e)$ and $T > T_0 $; this is to say that the relevant time $T_0$ depends on the trajectory in hand $(u,u_t)$ {\bf and} the equilibrium $e \in W$ to which it converges.
\end{lemma}
Note the slightly modified structure of the proof from that of Step 3 in {\bf Part I}.
\begin{proof}
We argue by contradiction. Restricting to the absorbing ball, there exists a sequence $z_n$ such that $E(z_n(t)) \leq R^2 $ and by boundedness of $lot(z_n)$,
$$\frac{lot(z_n)}{\int_s^T ||z_{n,t}||^2_0 } \rightarrow \infty\, ,\qquad\mbox{and so,}\qquad \int_s^T ||z_{n,t}||_0^2ds   \rightarrow 0 $$
and, consequently, $z_n \rightharpoonup_* z\text{ in }L^{\infty}(s,T; H_*^2)$ and
$z_{n,t}\rightharpoonup_* z_t\text{ in }L^{\infty}(s,T; L^2(\Omega))$; moreover $z_t\equiv 0$, so $z$ satisfies (weakly) on $H^2_*$ similarly to before:
$$\Delta^2z + P z_{xx}  - S ||(z+e)_{x} ||^2_0 z_{xx} + S [||e_{x}||^2_0 -||(z+e)_x||^2_0 ] e_{xx} =\alpha z_y.$$
{Since $e$ is stationary the above implies that $z+e$ is also a stationary point. By  \eqref{escape} along with weak convergence and lower semicontinuity of the energy, $$\int_{T-1}^T E(z)dt = E(z) \leq \epsilon \leq \epsilon_0$$ where $\epsilon_0 $ has been selected (as above, by the isolation hypothesis) so that there is no other equilibrium  with $E\big(S_t(y_0)-(e,0)\big) \leq \epsilon_0$.
From this we infer that for the limit point, $z\equiv 0 $.}
{\begin{remark}\label{R}
Note that in the case when $e$ is unique in $W$, then the conclusion that $z\equiv 0$ follows at once without the  necessity of assuming  convergence to an equilibrium at the outset of the proof.
\end{remark}}

Our next step is the rescaling argument which will yield the contradiction.
{We set
$ v_n \equiv \frac{z_n}{lot(z_n)^{1/2} } $}
and note that we have just shown $(z_n,z_{n,t}) \rightarrow (0,0)$ in $Y$.
We have
\begin{equation}\label{seq}
lot(v_n ) \equiv 1 , ~\text{ and }~\int_s^T ||v_{n,t}||^2 ds \rightarrow 0
\end{equation}
From  the rescaled  observability inequality (dividing \eqref{obs2} by $lot(z)$) we also have that
$$E(v_n(t)) \leq M,~~\forall~t>{T_0},$$
so we have weakly convergent subsequence (denoted by the same index $n$)
$$(v_n, v_{n,t}) \rightharpoonup_* (v, v_t )\text{ in }~L^{\infty}\big(s,T;\ H_*^2 \times L^2(\Omega)\big)$$
and, combining with \eqref{seq},
$$v_n \rightharpoonup_*~ v ~\text{ in }~ L^{\infty} (s,T; H^2(\Omega)), ~~\text{and} ~v_{n,t} \rightharpoonup_*~ 0 ~\text{ in }~ L^2(s,T;L^2(\Omega)).$$

From \eqref{zeq}, we have that $v_n$ satisfies (weakly) the following equation:
\begin{equation}\label{veq}
v_{n,tt} + k v_{n,t} + \Delta^2v_n  + [P - S ||z_{n,x} + e_{x}  ||^2 ]v_{n,xx} + \frac{1}{lot(z_n)} S[ ||e_{x}||^2 -||u_x||^2 ]e_{xx} ] = \alpha v_{n,y}.
\end{equation}
Rewriting the difference of squares, we have:
\begin{equation}\label{veq1}
v_{n,tt} + k v_{n,t} + \Delta^2v_n  + [P - S ||z_{n,x} + e_{x}  ||^2 ]v_{n,xx} -  S( v_{n,x}, e_{x} + z_{n,x} + e_{x}) e_{xx} = \alpha v_{n,y}.
\end{equation}
Passing with the limit on the weak form of the equation and exploiting the zero limits for $z_n$ and $v_{n,t}$ as before gives a linearization about $e$:
\begin{equation}\label{veq2}
 \Delta^2v  + [P - S || e_{x}  ||^2 ]v_{xx} - 2S ( v_x, e_{x})  e_{xx}  = \alpha v_{y}
\end{equation}
The assumption on hyperbolicity of the equilibrium $e$ implies immediately that $v \equiv 0 $.

Thus $v\equiv 0 $ and and by compactness, $lot(v_n) \rightarrow lot(v) =0$ which contradicts that $lot(v_n) \equiv 1$. Hence the desired estimate in Lemma \ref{thisone} holds.
\end{proof}

\noindent [\textbf{Step 4}] Combine Step 3 and Step 2 to obtain the observability-type inequality:
$${V}(T) +\int_s^T E_z(t) dt \leq c\left[ k   \int_s^T \|z_t\|_0^2dt\right].$$

\noindent [\textbf{Step 5}]
From the  balance identity for $V$ in \eqref{Vbalance}, we have
$${V}(t)+k\int_s^t\|z_t(\tau)\|_0^2ds\leq{V}(s)+\frac{1}{2k}\alpha^2\int_s^t\|z_y(s)\|_0^2ds+\frac{k}{2}  \int_s^t\|z_t(s)\|_0^2ds$$

As before, if $k^{-1}\alpha^2$ is sufficiently small, the last term is absorbed by the integrated quantity
$${V}(T) +\int_0^T V(t) dt  \leq  C[{V}(0) -{V}(T)],$$
which gives ${V}(T) \leq \frac{C}{C+1} {V}(0)$, and hence exponential decay, as desired.
{\begin{remark} Note here that, in general, $C$ -- which dictates the rate of decay -- depends on $T_0, e, S_t(y_0)$, which is to say the trajectory, the equilibrium to which it converges, and the ``no-escape" time associated to $e$. In general, one would not expect any uniformity across the set $W$, which is why Theorem \ref{th:main2a} is phrased as it is. In the case when $W$ is {finite}, in addition to isolated and hyperbolic, one can ascribe some uniformity to the decay rate (by choosing the minimal such value) and the critical $q$ parameter (controlling $\alpha^2k^{-1}$), again by choosing $q=\min_{e \in W} q_e$.
\end{remark}}

\section{Non-triviality of the attractor: proof of Theorem \ref{modalsol}}\label{modal}

The proof of Theorem \ref{modalsol} is organized as follows. Stationary solutions of \eqref{plate} (with $g=0$) solve the problem
\begin{empheq}{align}\label{stat-P}
\left\{
\begin{array}{rl}
\Delta^2 u + \left(P -S\left[\int_\Omega u_x^2\right]\right)u_{xx}= \alpha u_y  &\textrm{in }\Omega\times(0,T)\\
u = u_{xx}= 0 &\textrm{on }\{0,\pi\}\times[-\ell,\ell]\\
u_{yy}+\sigma u_{xx} = u_{yyy}+(2-\sigma)u_{xxy}= 0 &\textrm{on }[0,\pi]\times\{-\ell,\ell\}
\end{array}\right.
\end{empheq}
and we are first interested in (nontrivial) unimodal solutions of \eqref{stat-P}, see \eqref{unimodsol}. To this end,
we introduce the  related {linear} problem
\begin{empheq}{align}\label{eigenvalue}
\left\{
\begin{array}{rl}
\Delta^2 U -\mu U_{xx}= \alpha U_y  &\textrm{in }\Omega\\
U = U_{xx}= 0 &\textrm{on }\{0,\pi\}\times[-\ell,\ell]\\
U_{yy}+\sigma U_{xx} = U_{yyy}+(2-\sigma)U_{xxy}= 0 &\textrm{on }[0,\pi]\times\{-\ell,\ell\}\, .
\end{array}
\right.
\end{empheq}
The claimed properties about unimodal (stationary) solutions of \eqref{stat-P} are then obtained through \eqref{eigenvalue}. Finally,
infinitely many unimodal (time-dependent) solutions $V_{m,\alpha}$ of \eqref{plate} are constructed by means of the found solutions
of \eqref{stat-P}.\par
With a simple change of unknowns one obtains

\begin{lemma}\label{nonlin-to-lin}
Let $P\in\mathbb{R}$. If $\mu>-P$ is such that $(\alpha,U)$ ($U\neq0$) solves \eqref{eigenvalue}, then the function
$$u(x,y)=\sqrt{\frac{\mu+P}{S}}\, \frac{U(x,y)}{\|U_x\|_{0}}$$
is a nontrivial solution of \eqref{stat-P}.
\end{lemma}

Then, the existence of unimodal solutions of \eqref{stat-P} is then based on the next lemma.

\begin{lemma}\label{ll2}
Let $P\in\mathbb{R}$. For any integer $m$ there exists $\overline{\alpha}_m<0$ such that for all $\alpha<\overline{\alpha}_m$, the following assertions hold:\par\noindent
$\bullet$ There exists a unimodal solution $U_{m,\alpha}$ of \eqref{eigenvalue}, see \eqref{unimodsol}, having $m-1$ zeros in
the $x$-direction.\par\noindent
$\bullet$ There exists at least $m$ unimodal solutions $U_{1,\alpha},...,U_{m,\alpha}$ of \eqref{eigenvalue}; these solutions have from $0$
up to $m-1$ zeros in the $x$-direction.
\end{lemma}
\begin{proof}
To solve \eqref{eigenvalue}, we argue by separating variables, i.e.\ we seek a solution in the form
\begin{equation}\label{U}
U(x,y)=\psi(y)\sin(mx)\, .
\end{equation}
This amounts to solving the linear ODE
\begin{equation}\label{ODE}
\psi''''(y)-2m^2\psi''(y)-\alpha \psi'(y)+(m^4+\mu m^2)\psi(y)=0
\end{equation}
whose characteristic polynomial is $h_m(z)= z^4 -2m^2 z^2-\alpha z +m^4+\mu m^2$. When $\mu=\alpha=0$, we have $h_m(z)=(z^2-m^2)^2$ whose
graph is W-shaped with the two global minima
at $z=\pm m$ where $h_m(\pm m)=0$. If we increase $\mu$, the graph is shifted upwards and there are no real solutions of $h_m(z)=0$. Then we decrease $\alpha$ so that the graph starts leaning down
on the left of the origin and up on the right. There exists a unique critical negative value of $\alpha$, given by
$$
\alpha_m(\mu)=-\frac{4m}{3\sqrt3 }\Big(\sqrt{4m^2+3\mu}-2m\Big)\sqrt{m^2+m\sqrt{4m^2+3\mu}}<0\, ,
$$
for which the graph is tangent to the $z$-axis, namely the equation $h_m(z)=0$ has a double zero for some $z<0$. When $\alpha < \alpha_m(\mu)$, the global minimum of $h_m$ becomes negative and
there exist two negative real solutions of $h_m(z)=0$, say $z_1<z_2<0$, the remaining solutions $z_3$ and $z_4$ being complex (and, obviously, conjugated). Also, $z_1\to -\infty$ and $z_2\to 0$ as $\alpha\to-\infty$. Moreover, $z_3=\bar z_4$, $Re[z_3]>0$, $Re[z_3]\to\infty$ as $\alpha\to-\infty$ and $Im[z_3]\to\infty$ as $\alpha \to-\infty$.
We found the explicit expressions of the $z_i$'s due to Ferrari-Descartes by using Mathematica and, subsequently, we checked them by hand. 
Hence, when $\alpha < \alpha_m(\mu)$ the general solution of \eqref{ODE} reads
$$\psi(y) = A_1\exp(z_1y) + A_2\exp(z_2y) + \exp(Re[z_3]y)(A_3 \cos(Im[z_3]y)+ A_4 \sin(Im[z_3]y))\, .$$
By imposing that the function $U$ in \eqref{U} satisfies the boundary conditions in \eqref{eigenvalue}, we find the four conditions
$$\psi''(\pm \ell)-\sigma m^2\psi(\pm \ell) = \psi'''(\pm \ell) -(2-\sigma)m^2\psi'(\pm \ell)=0$$
that constitute a linear $4\times4$ algebraic system of the unknowns $A_1,...,A_4$. Also the explicit form of the determinant $D=D(m,\mu,\alpha)$
of this system was computed by using Mathematica and checked by hand, it is a function depending on $m$, $\mu$, $\alpha$. Then, as is standard in eigenvalue problems,
\begin{center}
there exists a nontrivial solution $U$ of \eqref{eigenvalue} of the form \eqref{U} if and only if $D(m,\mu,\alpha)=0$.
\end{center}
At this point, explicit computations became even more difficult and we merely proceeded with Mathematica, with no hand control.
The condition $D(m,\mu,\alpha)=0$ defines implicitly an analytic negative function $\alpha=\Phi(\mu,m)$ whose absolute value $|\Phi|$
is numerically seen to be strictly increasing with respect to both $m$ and $\mu$ with $\displaystyle\lim_{\mu\to\infty}|\Phi(\mu,m)|=\infty$.
For fixed $m=1,2,3,4,5$, in Figure \ref{plots} we report the plot of the functions $\mu\mapsto|\Phi(\mu,m)|$.
\begin{figure}[h!]
\begin{center}
\includegraphics[width=70mm]{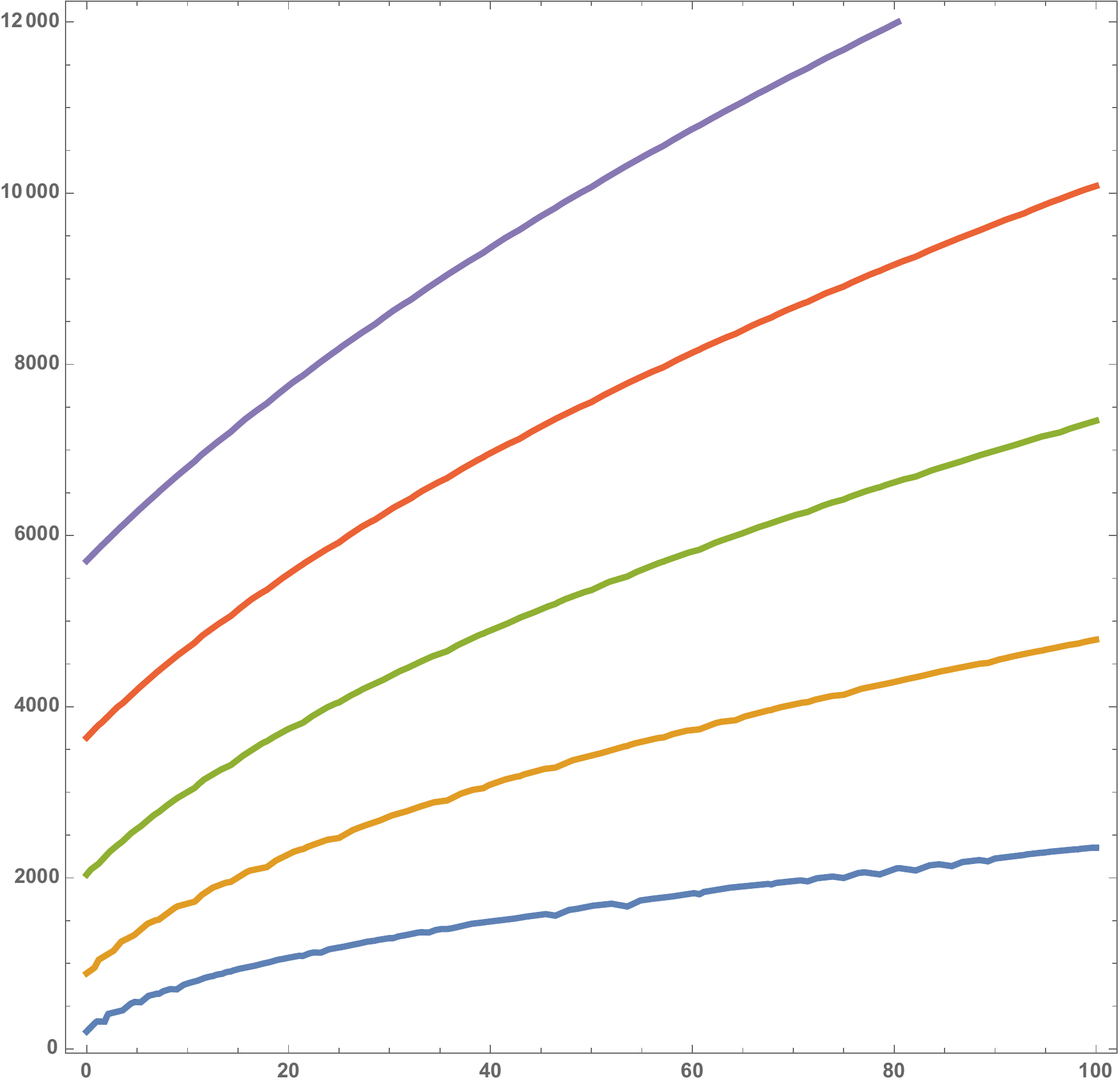}
\caption{Graphs of the maps $\mu\mapsto|\Phi(\mu,m)|$ for $m=1,2,3,4,5$ (from bottom to top).}\label{plots}
\end{center}
\end{figure}
It is apparent that they are strictly increasing and divergent as $\mu\to\infty$. Then we fixed $\mu$ and we considered the map
$m\mapsto|\Phi(\mu,m)|$: it also turned out to be increasing and divergent as $m\to\infty$: in Figure \ref{plots} we see that
$$|\Phi(\mu,1)|<|\Phi(\mu,2)|<|\Phi(\mu,3)|<|\Phi(\mu,4)|<|\Phi(\mu,4)|\qquad\forall \mu\in(0,100)$$
and these inequalities continue for all $m$.\par
The above numerical arguments show that for a given integer $m$, if $\alpha$ is sufficiently negative (say $\alpha<\overline{\alpha}_m<0$)
then $D(m,\mu,\alpha)=0$ for some $\mu>0$. As a consequence, there exists a nontrivial solution $U$ of \eqref{eigenvalue} of the form \eqref{U}
which has $m-1$ zeros in the $x$-direction. This proves the first item.\par
Moreover, since $m\mapsto|\Phi(\mu,m)|$ is increasing, for the same $\alpha$ and for any $i=1,...,m-1$ we may find $\mu_i>\mu$ such
that $D(i,\mu_i,\alpha)=0$ so that there exists a nontrivial solution $U$ of \eqref{eigenvalue} of the form \eqref{U} (with $m$ replaced by $i$)
which has $i-1$ zeros in the $x$-direction. This proves the second item and completes the proof of the lemma.
\end{proof}

The first two items in Theorem \ref{modalsol} (existence and multiplicity of unimodal solutions of \eqref{plate}) are a direct
consequence of Lemmas \ref{nonlin-to-lin} and \ref{ll2}.\par\smallskip

Next we build an evolution unimodal solution to \eqref{plate} with $g=0$. Assume that $\alpha < \overline{\alpha}_m$ so that, by Lemmas
\ref{nonlin-to-lin} and \ref{ll2}, there exists a (stationary) solution of \eqref{stat-P} of the kind $U(x,y)=\psi(y)\sin(mx)$ and
$$\Delta^2 U - m^2\left(P -S\left[\int_\Omega U_x^2\right]\right)U= \alpha \psi'(y)\sin(mx).$$
Now $W(x,y,t):=\phi(t)U(x,y)$ solves the evolution equation \eqref{plate} with $g=0$ if and only if
$$U\ddot{\phi}+k U\dot{\phi}+ \Delta^2 U \phi - \left[P - S \phi^2\int_\Omega U_x^2\right] m^2\phi U=\phi \alpha \psi'(y)\sin(mx), $$
if and only if
$$U\ddot{\phi}+kU\dot{\phi}+ \Delta^2 U \phi - \left[P - S \phi^2\int_\Omega U_x^2\right] m^2\phi U=
\phi\left(\Delta^2 U - m^2\left(P -S\left[\int_\Omega U_x^2\right]\right)U\right).$$
After simplifying this equation, we infer that
$W(x,y,t)$ solves the evolution equation \eqref{plate} with $g=0$ if and only if
$$U(x,y)\left(\ddot{\phi}+k\dot{\phi}+ Sm^2 (\phi^3-\phi)\int_\Omega U_x^2\right)=0.$$
Let us set
$$Sm^2\int_\Omega U_x^2=:m^4R^2.$$
We finally deduce that
$W(x,y,t)$ solves the evolution equation \eqref{plate} with $g=0$ if and only if
$\phi=\phi(t)$ is a solution of the damped Duffing equation
\begin{equation}\label{duffing}
\ddot{\phi}+k\dot{\phi}+ (\phi^3-\phi)m^4R^2=0.
\end{equation}

Then we notice that any solution of \eqref{duffing} satisfies the identity
$$
\frac{d}{dt}\left(\frac12 \dot\phi (t)^2+m^4R^2\big(\frac{\phi(t)^4}{4}-\frac{\phi(t)^2}{2}\big)\right)=-k\dot\phi (t)^2\, .
$$
We infer that all the solutions of \eqref{duffing} are globally bounded, and therefore $\phi(t)$ tends to a constant solution of \eqref{duffing}, that is, $\ds \lim_{t\to\infty}\phi(t)\in\{-1,0,1\}$. In turn, this means that one of the following facts occurs:
$$
W(x,y,t)\to0\, ,\qquad W(x,y,t)\to U(x,y)\, ,\qquad W(x,y,t)\to -U(x,y)\qquad\mbox{as }t\to\infty\, .
$$
In particular, if
\begin{equation}\label{conditionneg}
2 \dot\phi (0)^2+ m^4R^2(\phi(0)^4-2\phi(0)^2)<0\, ,
\end{equation}
then $W(x,y,t)\not\to0$ as $t\to\infty$ and $W(x,y,t)$ necessarily tends to either $U(x,y)$ or $-U(x,y)$. This proves the statement about evolution unimodal solutions. \par

\begin{remark}\label{yto-y}
Theorem \ref{modalsol} explains how the bifurcation from the trivial solution occurs, arising from $\alpha=\overline{\alpha}_m$ as $\alpha$ decreases. Or, backwards, when
$\alpha \uparrow\overline{\alpha}_m$, the norm of the stationary solution tends to $0$. Moreover, Theorem \ref{modalsol} enables us to construct heteroclinic solutions as follows. Take a sequence of initial
values $(\phi(0),\dot{\phi}(0))=(1/n,0)$ so that \eqref{conditionneg} holds. These data tend to 0 as $n\to\infty$ while, as $t\to\infty$, the corresponding solution of \eqref{plate} tends to $U$.

If $u(x,y,t)$ solves \eqref{plate} for some $\alpha<0$ then $u(x,-y,t)$ solves \eqref{plate} when $\alpha$ is replaced by $-\alpha>0$. This also occurs for the stationary problem \eqref{stat-P} and for the eigenvalue problem
\eqref{eigenvalue}. This shows that one can reflect vertically Figure \ref{plots} and have a picture for all $\alpha \in\R$.
Moreover, by arguing as for \eqref{lowerbeta}, one finds that $|\alpha|>\frac{\lambda_1-P}{\lambda_1} \sqrt{2\lambda_1(1-\sigma^2)}$ is a necessary condition for the existence of nontrivial solutions to \eqref{eigenvalue}.
This serves as a lower bound for the curve in Figure \ref{plots}.
\end{remark}

\section{Construction of determining functionals: proof of Theorem \ref{defectcorollary}}\label{construct}

	We prove a more general result than Theorem \ref{defectcorollary}, in the setting of a {determining set of functionals} (note the construction in \cite[Theorem 7.2]{dissipative}, as well as \cite[Section 7.9.4]{springer} and \cite{quasi}). This abstract theory allows us to show that any set of functionals satisfying a particular smallness condition will be determining.
Let us recall the notion of determining set.

\begin{definition}[Determining set]	
Let $\mathscr L=\{l_j~:~j\in I\}$ be a set of continuous, linear functionals on $H_*^2$, where $I$ is some index set. We say that $\mathscr L$ is a {determining set of functionals} if for any two trajectories $S_t(y^i)=(u^i(t),u_t^i(t)),~i=1,2$, we have that
$$\lim_{t \to \infty} \|S_t(y^1)-S_t(y^2)\|_{Y}^2 =0\qquad\mbox{whenever}\qquad
\lim_{t \to \infty} | l_j(u^1(t)-u^2(t))|^2 = 0,\ \forall\,j\in I.
$$
\end{definition}

Roughly speaking, a collection of functionals is asymptotically determining if evaluation on these functionals (as $t\to\infty$) is sufficient
to distinguish trajectories. As discussed above, in most cases, we are looking for a {finite} set $\mathscr L$ that is asymptotically determining for $(S_t,Y)$.
 \begin{definition}[Completeness Defect]
 Let $\mathscr L=\{l_i\}_{i=1}^N$ be a finite set of linear functionals on $H_*^2$. The completeness defect of $\mathscr L$ on $H^2_*$, with respect to $H^s(\Omega)$ ($0 \le s <2$), is defined by
	\begin{equation}
	\varepsilon_{\mathscr L,s} := \varepsilon_{\mathscr L}(H_*^2,H^s(\Omega))=  \sup_{\{\|w\|_{H^2_*} \le 1\}}\big\{\|w\|_{H^s(\Omega)}~:~ l_j(w)=0~~~\forall~j=1,...,N\big\}.\end{equation}
\end{definition}

With this notion at hand, we can present the main result on determining functionals for \eqref{plate}.
	
\begin{theorem}[Determining Functionals]\label{defecttheorem}
Let $k,S>0$, $\alpha,P \in \mathbb R$  and $(S_t, Y)$ be as above. There exists a number $\varepsilon_*>0$ such that if $\mathscr L$ is a set
of continuous, linear functionals on $H_*^2$ with
$\varepsilon_{\mathscr L,0} \le \varepsilon_*$, then $\mathscr L$ is a determining set of functionals for $(S_t,Y)$.
\end{theorem}
We first prove a key lemma.
	\begin{lemma}\label{doya}
	Let $\mathscr L=\{l_i\}_{i=1}^N$ be a finite set of linear functionals on $H_*^2$ and $0<\eta\le 2$. Then, the exists $C({\mathscr L},\eta)>0$ such that for any $v \in H^2_*$, we have
	\begin{equation}\label{lemmause}\|v\|_{2-\eta} \le \varepsilon_{\mathscr L,2-\eta}\|v\|_{H^2_*}+C({\mathscr L},\eta)\max_{j = 1,...,N} |l_j(v)|.\end{equation}
	\end{lemma}
	\begin{proof}
	Let $\{ e_j~:~j=1,...,N\}$ be an orthonormal system for $\mathscr L$ i.e. $l_j(e_i)=\delta_{ij}$. Given $v \in H^2_*$, we set
	$w = v-\sum_{j=1}^N l_j(v)e_j$. {Clearly, $l_j(w)=0$ for $j=1,...,N$ and hence, directly from the definition of $\varepsilon_{\mathscr L,2-\eta}$, we have
	$$\|w\|_{2-\eta} \le \varepsilon_{\mathscr L,2-\eta}\|w\|_{H_*^2}.$$}
	Then we write
\begin{eqnarray*}
\|v\|_{2-\eta} & \le & \|v-w\|_{2-\eta}+\|w\|_{2-\eta}\le\|v-w\|_{2-\eta} + \varepsilon_{\mathscr L,2-\eta}\|w\|_{H_*^2}\\
& \le & \|v-w\|_{2-\eta}+ \varepsilon_{\mathscr L,2-\eta}\|v-w\|_{H_*^2}+ \varepsilon_{\mathscr L,2-\eta}\|v\|_{H_*^2}\\
& \le &  C({\mathscr L},\eta)\max_{j = 1,...,N} |l_j(v)| + \varepsilon_{\mathscr L,2-\eta}\|v\|_{H^2_*},
\end{eqnarray*}
for some $C({\mathscr L},\eta)>0$.
	\end{proof}

\begin{proof}[Proof of Theorem \ref{defecttheorem}]
	
	Let $S_t(y^i)=(u^i(t),u^i_t(t))$ be two trajectories for $y^1, y^2 \in \mathscr B \subseteq  Y$.
We claim that, if $\varepsilon_{\mathscr L,2-\eta}$ is sufficiently small, then: 
\begin{equation}\label{limlim}
\lim_{t\to\infty}|l_j(u^1(t)-u^2(t))|^2 = 0,~~\forall~~j=1,...,N,\quad\Longrightarrow\quad\lim_{t \to \infty} \|S_t(y^1)-S_t(y^2)\|_{Y}^2 =0.
\end{equation}
 Indeed, suppose that the assumption in \eqref{limlim} holds and note that this is equivalent to
 \begin{equation}\label{needthis}\mathscr S(t) := \sup_{s \in [t,t+\tau]} \max_j |l_j(u^1(s)-u^2(s))|^2=0,~~t\to \infty.\end{equation}
 In the sequel, $C$ denotes a positive constant independent of the trajectories and which may vary from line to line.
 The quasi-stability estimate \eqref{qs*}, where $(z(t),z_t(t))=S_t(y^1)-S_t(y^2)$, and the semigroup property, yield the inequality
 \begin{equation}\label{usingthisone}
 \|S_{t+\tau}(y^1)-S_{t+\tau}(y^2)\|_{Y}^2 \le C(e^{-\sigma \tau}\|S_t(y^1)-S_t(y^2)\|_Y+\sup_{t \le s \le t+\tau} \|z(t)\|_{2-\eta}^2)
 \end{equation}
 With Young's inequality, we have from \eqref{lemmause} for any $b>0$, the exists $C=C(b)>0$ such that
 $$\|v\|_{2-\eta}^2 \le (1+b)\varepsilon_{\mathscr L,2-\eta}^2\|v\|_{H_*^2}^2+C\max_{j=1,...,N} |l_j(v)|^2.$$
 With the Lipschitz estimate on $S_t$ in \eqref{dynsys}, we obtain from above
 $$\sup_{t \le s \le t+\tau} \|z(s)\|_{2-\eta}^2 \le [(1+b)\varepsilon_{\mathscr L,2-\eta}Ce^{a_R\tau}]\|S_t(y^2)-S_t(y^2)\|_{Y}^2+C(\mathscr L,b,\eta)\mathscr S(t).$$
 From this estimate, we invoke \eqref{usingthisone} to obtain
 $$\|S_{t+\tau}(y^1)-S_{t+\tau}(y^2)\|_{Y}^2 \le \Upsilon \|S_t(y^1)-S_t(y^2)\|_{Y}^2 + C(\mathscr L,b,\eta)\mathscr S(t),$$
 with $\Upsilon=\mathscr C(\sigma,\mathscr B)[(1+b)\varepsilon_{\mathscr L,2-\eta}e^{a_R\tau}+e^{-\sigma \tau}]$. For any $b>0$, and a sufficiently large $\tau >0$, by taking {$\varepsilon_{\mathscr L,2-\eta}$ sufficiently small}, we guarantee  $\Upsilon<1$. Then, again from the semigroup property, we can iterate on intervals of size $\tau$ to obtain
$$\|S_{t_0+n\tau}(y^1)-S_{t_0+n\tau}(y^2)\|_{Y}^2 \le \Upsilon^n\|S_{t_0}(y^1)-S_{t_0}(y^2)\|_{Y}^2 +C\sum_{m=0}^{n-1}\Upsilon^{n-m-1}\mathscr S(t_0+m\tau).$$
 From here, taking $n\to \infty$, we obtain from \eqref{needthis} the desired conclusion in \eqref{limlim} and the proof of Theorem \ref{defecttheorem} is complete, once we note that  {$\varepsilon_{\mathscr L,0}$ controls $\varepsilon_{\mathscr L,2-\eta}$} as in \eqref{control}.
 
Indeed, we obtain the relation between $\varepsilon_{\mathscr L,2-\eta}$ and $\varepsilon_{\mathscr L,0}$ through interpolation. First, standard Sobolev interpolation yields:
{\begin{equation} \|u\|_{2-\eta} \le \|u\|_0^{\eta/2}\|u\|_{2}^{1-\frac{\eta}{2}}\le c\|u\|_0^{\eta/2}\|u\|_{H^2_*}^{1-\frac{\eta}{2}}.\end{equation}}Then from	\cite[(3.3.9) - p.123]{quasi} with {$V= H^2_*$, $W=H^{2-\eta}(\Omega)$, and $X = L^2(\Omega)$} where $\theta=\eta/2$ and $a_{\theta}=c$ (the constant related to norm equivalence above), we infer that
\begin{equation} \label{control} [\varepsilon_{\mathscr L,2-\eta}]^{2/\eta}\le  c^{2/\eta}\varepsilon_{\mathscr L,0}\le c^{\frac{4}{\eta(2-\eta)}}\big[\varepsilon_{\mathscr L}\big(H^{2-\eta}(\Omega),L^2(\Omega)\big)\big]^{2/(2-\eta)}.\end{equation}
Taking $\varepsilon_{\mathscr L,0}<\varepsilon_*$ sufficiently small with respect to the control in \eqref{control} then completes the proof.\end{proof}

The example of central interest here is that of {determining modes}. Let $\{e_j\}$ be the eigenfunctions of $A$ on $H_*^2$.
Then, for the set
\begin{empheq}{equation}\label{thesemodes}\mathscr L=\{ l_j ~:~l_j(w)=(w,e_j),~~j=1,...,N\},\end{empheq}
define the Fourier approximation $R_{\mathscr L}:H_*^2\to H_*^2$ by
$$R_{\mathscr L}(w)=\sum_{j=1}^N l_j(w)e_j.$$
Then $R_{\mathscr L}$ approximates in $L^2(\Omega)$, in that there exists $C,\alpha>0$ such that
	\begin{empheq}{equation}\label{goodapprox}\|w-R_{\mathscr L}(w)\|_{0} \le Ch^{\alpha},\end{empheq}
	for all $w\in H^2_*$ with $\| w\|_{H^2_*}\le 1$, and any $h(N)>0$ sufficiently small. Specifically, in this case, we have that \begin{empheq}{equation}\varepsilon_{\mathscr L,0}=\varepsilon_{\mathscr L}(H_*^2,L^2(\Omega)) \le c/N,\end{empheq} for some $c$, and for all $N$ sufficiently large; see \cite[Section 3.3]{quasi} for further details.
We can then apply Theorem \ref{defecttheorem} to obtain, as a consequence, Theorem \ref{defectcorollary}.

\appendix
\addcontentsline{toc}{section}{APPENDICES}
\section*{Appendices}
\section{Nodes of oscillating modes and spectral analysis}\label{spectral}

The Federal Report \cite{ammvkwoo} makes a detailed description of the oscillations seen prior to the Tacoma collapse. In particular, we learn that in the days before the collapse:\par
$\diamondsuit$ {One principal mode of oscillation prevailed ... the modes of oscillation frequently changed}\;\par
$\diamondsuit$ {Altogether, seven different motions have been definitely identified on the main
span of the bridge ... These different wave actions consist of motions from the simplest, that of no nodes, to the most complex, that of seven modes}.\par
On the other hand, the day of the collapse, the following was observed:\par
$\diamondsuit$ {prior to 10:00 A.M.\ on the day of the failure, there were no recorded instances of the oscillations being otherwise than the two
cables in phase and with no torsional motions};\par
$\diamondsuit$ {the bridge appeared to be behaving in the customary manner ... these motions, however, were considerably less than had occurred many times before};\par
$\diamondsuit$ {the only torsional mode which developed under wind action on the bridge or on the model is that with a single node at the center
of the main span}.\par\medskip

The above demonstrates the importance given to the nodes of the oscillating bridge modes. In this respect, we refer to Drawing 4 in \cite{ammvkwoo}:
it is an attempt to classify the observed modes of oscillations. This is why the analysis of unimodal solutions (as in Section \ref{modal}) is relevant to us.

We recall here some results about the eigenvalue problem

\begin{equation}\label{eq:eigenvalueH2L2}
\left\{
\begin{array}{rl}
\Delta^2 w = \lambda w& \textrm{in }\Omega \\
w = w_{xx} = 0& \textrm{on }\{0,\pi\}\times[-\ell,\ell]\\
w_{yy}+\sigma w_{xx} = 0& \textrm{on }[0,\pi]\times\{-\ell,\ell\}\\
w_{yyy}+(2-\sigma)w_{xxy} = 0& \textrm{on }[0,\pi]\times\{-\ell,\ell\}
\end{array}
\right.
\end{equation}
which can be equivalently rewritten as $a(w,z)=\lambda (w,z)$ for all $z\in H^{2}_*(\Omega)$. Here we will take
\begin{empheq}{align}\label{lsigma}
\ell=\frac{\pi}{150}\ ,\quad\sigma=0.2\,
\end{empheq}
with the relevant Poisson ratio $\sigma$ in mind for a suspension bridge (a mixture of iron and concrete) and the measures of the collapsed TNB.
By combining  \cite{bebuga,bfg1,2015ferreroDCDSA,bookgaz}, we obtain this statement.

\begin{proposition}\label{spectrum}
The set of eigenvalues of \eqref{eq:eigenvalueH2L2} may be ordered in an increasing sequence of strictly positive numbers diverging to $+\infty$ and
any eigenfunction belongs to $C^\infty(\overline\Omega)$. The set of eigenfunctions of \eqref{eq:eigenvalueH2L2} is a complete system in $H^2_*$.
Moreover, an eigenfunction associated to an eigenvalue $\lambda_j$ has the form
$$\phi_j(y)\sin (m_j x),$$
where $\phi_j$ is either odd or even and $m_j>0$ is an integer. 

\end{proposition}
A more precise statement (including the explicit form of $\phi_j$) is given in \cite{bbg1,bebuga,bfg1,2015ferreroDCDSA}. 
Moreover, we know that for our configuration on $\Omega = (0,\pi)\times (-\ell,\ell)$
\begin{equation} \label{lambda}
\lambda_1:=\mu_{1,1}\, =\, \min_{v\in H^2_*}\ \frac{\|v\|_{H^2_*}^2}{\|v_x\|_{0}^2}\, =\, \min_{v\in H^2_*}\ \frac{\|v\|_{H^2_*}^2}{\|v\|_{0}^2}\qquad
\mbox{and}\qquad\min_{v\in H^2_*}\ \frac{\|v_x\|_{0}^2}{\|v\|_{0}^2}=1\, ,
\end{equation}
which yields the following embedding inequalities, see \cite[p. 3060, Eq. (9)]{bongazmor},
\begin{equation}\label{embedding}
\|v\|_{0}^2\le\|v_x\|_{0}^2\, ,\quad\lambda_1\|v\|_{0}^2\le\|v\|_{H^2_*}^2\, ,\quad\lambda_1\|v_x\|_{0}^2\le\|v\|_{H^2_*}^2\qquad
\forall v\in H^2_*\, .
\end{equation}

\section{Long-time behavior of dynamical systems}\label{appendix3}

We recall here notions and results from the theory of dissipative dynamical systems.
We say that the dynamical system $(S_t,Y)$ is {asymptotically smooth} if for any
bounded, forward invariant set $D$ there exists a compact set $K
\subset \overline{D}$ such that $~\ds
\lim_{t\to+\infty}d_{Y}\{S_t(D)~|~K\}=0$.
A closed set $B \subset Y$ is {absorbing} if for any bounded set $D \subset Y$ there exists a $t_0(D)$ such that $S_t(D) \subset B$ for all $t > t_0$. If $(S_t,Y)$ has a bounded absorbing set, it is said to be {ultimately dissipative}.
We will use a key theorem from \cite[Chapter 7]{springer} to establish the attractor and its characterization.
\begin{theorem}
\label{dissmooth} A dissipative and asymptotically smooth dynamical system $(S_t,Y)$ has a unique compact global attractor $\textbf{A} \subset \subset Y$ that is connected,  characterized by the set of all bounded, full trajectories.
\end{theorem}

We now proceed to discuss the theory of quasi-stability of Chueshov and Lasiecka \cite{springer,quasi}.
\begin{condition}\label{secondorder} Consider dynamics $(S_t,Y)$ where $Y=X \times Z$ with $X,Z$ Hilbert, and $X$ compactly embedded into $Z$.  Suppose $y= (x,z) \in Y$ with $S_t(y) =(x(t),x_t(t))$ and  $x \in C^0(\mathbb R_+,X)\cap C^1(\mathbb R_+,Z)$.
\end{condition}
 Condition \ref{secondorder} restricts our attention to second order, hyperbolic-like evolutions.
\begin{condition}\label{locallylip} Suppose $S_t\in Lip_{loc}(Y)$, with Lipschitz constant $a(t)\in L^{\infty}_{loc}(0,\infty)$:
\begin{equation}\label{specquasi}
\|S_t(y_1)-S_t(y_2)\|_Y^2 \le a(t)\|y_1-y_2\|_Y^2.
\end{equation}
\end{condition}
\begin{definition}\label{quasidef}
With conditions \ref{secondorder} and \ref{locallylip} in force, suppose that the dynamics $(S_t,Y)$ admit the following estimate for $y_1,y_2 \in B \subset Y = X \times Z$:
\begin{equation}\label{specquasi*}
\|S_t(y_1)-S_t(y_2)\|_Y^2 \le e^{-\gamma t}\|y_1-y_2\|_Y^2+C_q\sup_{\tau \in [0,t]} \|x_1-x_2\|^2_{Z_*}, ~~\text{ for some }~~\gamma, C_q>0,
\end{equation}
where ~$ X \subseteq Z_* \subseteq Z$ and the last embedding is compact. Then $(S_t,Y)$ is {quasi-stable} on $B$.
\end{definition}

We now run through a handful of consequences of a system satisfying Definition \ref{quasidef} above for dynamical systems $(S_t,Y)$ satisfying Condition \ref{secondorder}
\cite[Proposition 7.9.4]{springer}.
\begin{theorem}\label{doy}
If a dynamical system $(S_t,Y)$ satisfying Conditions  \ref{secondorder} and \ref{locallylip} is quasi-stable on every bounded, forward invariant set $ B \subset Y$, then $(S_t,Y)$ is asymptotically smooth. If, in addition, $(S_t,Y)$ is ultimately dissipative, then by Theorem \ref{dissmooth} there is a compact global attractor $\bA \subset \subset Y$.
\end{theorem}

The theorems in \cite[Theorem 7.9.6 and 7.9.8]{springer} provide the following result for improved properties of the attractor $\bA$ if the quasi-stability estimate can be shown {on} $\bA$. If Theorem \ref{doy} is used to construct the attractor, then Theorem \ref{dimsmooth} follows immediately; this is not always possible \cite{delay,HLW}.
\begin{theorem}\label{dimsmooth}
If a dynamical system $(S_t,Y)$ satisfying conditions  \ref{secondorder} and \ref{locallylip} possesses a compact global attractor $ \bA \subset \subset Y$, and is quasi-stable on $\bA$, then $\bA$ has finite fractal dimension in $Y$ ($\text{dim}_f^Y\bA <+\infty$). Moreover, any full trajectory $\{(x(t),x_t(t))~:~t \in \mathbb R\} \subset \bA$ has the property that
$$x_t \in L^{\infty}(\mathbb R;X)\cap C^0(\mathbb R;Z);~~x_{tt} \in L^{\infty}(\mathbb R;Z),~\text{ with bound }~
\|x_t(t)\|^2_X+\|x_{tt}(t)\|_Z^2 \le C,$$
where the constant $C$ above depends on the ``compactness constant" $C_q$ in \eqref{specquasi*}.
\end{theorem}
\noindent Elliptic regularity can then be applied to the equation itself generating the dynamics $(S_t,Y)$ to recover  $x(t)$ in a norm higher than that of the state space $X$.

The following theorem relates fractal exponential attractors to quasi-stability: \begin{theorem}[Theorem 7.9.9 \cite{springer}]\label{expattract*}
Let Conditions  \ref{secondorder} and \ref{locallylip} be in force. Assume that the dynamical system $(S_t,Y)$ is ultimately dissipative and quasi-stable on a bounded absorbing set $ B$. Also assume there exists a space $\widetilde Y \supset Y$ so that $t \mapsto S_t(y)$ is H\"{o}lder continuous in $\widetilde Y$ for every $y \in  B$; this is to say there exists $0<\alpha \le 1$ and $C_{ B,T>0}$ so that
\begin{equation}\label{holder}\|S_t(y)-S_s(y)\|_{\widetilde Y} \le C_{ B,T}|t-s|^{\alpha}, ~~t,s\in\R_+,~~y \in  B.\end{equation} Then the dynamical system $(S_t,Y)$ possesses a generalized fractal exponential attractor $A_{\text{exp}}$ whose dimension is finite in the space $\widetilde Y$, i.e., $\text{dim}_f^{\widetilde Y} A_\text{exp}<+\infty$.
\end{theorem}
\begin{remark}
We forgo using boldface to describe $A_{exp}$ (in contrast to {\bf the} global attractor $\bA$) because exponential attractors are not unique.
In addition, owing to the abstract construction of the set $A_{\text{exp}} \subset X$, {boundedness} of $A_{\text{exp}}$ in any higher topology is not addressed by Theorem \ref{expattract*}. \end{remark}

\medskip\noindent
{\bf Acknowledgements.} The first author is supported by the Thelam Fund (Belgium), Research proposal FRB 2019-J1150080.
The second author is supported by PRIN from the MIUR and by GNAMPA from the INdAM (Italy).
The third and fourth authors acknowledge support from the National Science Foundation from DMS-1713506 (Lasiecka) and NSF
DMS-1907620 (Webster).

\end{document}